\documentclass[11pt]{amsart}
 
\usepackage{hyperref}
\usepackage{amssymb,amsmath,color}

\theoremstyle{plain}
\newtheorem{thm}{Theorem}[section]

\newtheorem{lem}[thm]{Lemma}

\newtheorem{cor}[thm]{Corollary}
\newtheorem{prop}[thm]{Proposition}

\theoremstyle{definition}
\newtheorem{defn}[thm]{Definition}

\theoremstyle{remark}

\newtheorem*{remark*}{Remark}

\numberwithin{equation}{section}

\newcommand{\cF}{{\mathcal{F}}}

\newcommand{\cH}{{\mathcal{H}}}

\newcommand{\cL}{{\mathcal{L}}}

\newcommand{\cO}{{\mathcal{O}}}
\newcommand{\cP}{{\mathcal{P}}}

\newcommand{\cT}{{\mathcal{T}}}

        \newcommand{\field}[1]{{\mathbb{#1}}}
        \newcommand{\NN}{\field{N}}
        \newcommand{\ZZ}{\field{Z}}
        
        \newcommand{\RR}{\field{R}}
        \newcommand{\CC}{\field{C}}

\newcommand{\tr}{\mbox{\rm tr}}

\allowdisplaybreaks

\begin{document}

\title[Semiclassical spectral analysis of Toeplitz operators]{Semiclassical spectral analysis of Toeplitz operators on symplectic manifolds: the case of discrete wells}

\author[Y. A. Kordyukov]{Yuri A. Kordyukov}
\address{Institute of Mathematics with Computing Centre\\
Ufa Federal Research Centre of 
         Russian Academy of Sciences\\
         112~Chernyshevsky str.\\ 450008 Ufa\\ Russia\\
ORCID: 0000-0003-2957-2873} \email{yurikor@matem.anrb.ru}

\thanks{Supported by the Russian Science Foundation,
project no. 17-11-01004.}

\subjclass[2000]{Primary 58J50; Secondary 53D50, 58J37}

\keywords{Bochner Laplacian, symplectic manifolds, semiclassical analysis, Berezin-Toeplitz quantization, eigenvalue asymptotics}

\begin{abstract}
We consider Toeplitz operators associated with the renormalized Bochner Laplacian on high tensor powers of a positive line bundle on a compact symplectic manifold. We study the asymptotic behavior, in the semiclassical limit, of low-lying eigenvalues and the corresponding eigensections of a self-adjoint Toeplitz operator under assumption that its principal symbol has a non-degenerate minimum with discrete wells. As an application, we prove upper bounds for low-lying eigenvalues of the Bochner Laplacian in the semiclassical limit.   
\end{abstract}

\date{}

 \maketitle

%\tableofcontents

\section{Preliminaries and main results}
\subsection{Introduction}
The Berezin-Toeplitz operator quantization is a quantization method for a compact quantizable symplectic manifold, which is a particularly effective version of geometric quantization. It was Berezin who recognized the importance of Toeplitz operators for quantization of K\"ahler manifolds in his pioneering work \cite{Berezin}. There are several approaches to Berezin-Toeplitz and geometric quantization (see, for instance, survey papers \cite{Ali,Englis,ma:ICMtalk,Schlich10}). For a general compact K\"ahler manifold, the Berezin-Toeplitz quantization was constructed by 
Bordemann-Meinrenken-Schli\-chen\-maier \cite{BMS94}, using the theory of Toeplitz structures of Boutet de Monvel and Guillemin \cite{BG}. In this case, the quantum space is the space of holomorphic sections of tensor powers of the prequantum line bundle over the K\"ahler manifold. In order to generalize the Berezin-Toeplitz quantization to arbitrary symplectic manifolds, one has to find a substitute for this quantum space. A natural candidate suggested by Guillemin and Vergne is the kernel of the spin$^c$ Dirac operator. The Berezin-Toeplitz quantization with such a quantum space was developed by Ma-Marinescu \cite{ma-ma:book,ma-ma08}. It is based on the asymptotic expansion of the Bergman kernel outside the diagonal obtained by Dai-Liu-Ma \cite{dai-liu-ma}. Another candidate suggested by Guillemin-Uribe \cite{Gu-Uribe} is the space of eigensections of the renormalized Bochner Laplacian corresponding to eigenvalues localized near the origin. In this case, the Berezin-Toeplitz quantization was recently constructed in \cite{ioos-lu-ma-ma,bergman}, based on Ma-Marinescu work: the Bergman kernel expansion from \cite{ma-ma08a} and Toeplitz calculus developed in \cite{ma-ma08} for spin$^c$ Dirac operator and K\"ahler case (also with an auxiliary bundle). We note also that Charles \cite{charles} proposed recently another approach to quantization of symplectic manifolds and Hsiao-Marinescu \cite{HM} constructed a Berezin-Toeplitz quantization for eigensections of small eigenvalues in the case of complex manifolds.

The simplest example of the Berezin-Toeplitz quantization is given by the Toeplitz operators on the Fock space. This quantization is related with the standard pseudodifferential calculus in the Euclidean space via the Bargmann transform. Therefore, for a general quantizable symplectic manifold, the Berezin-Toeplitz quantization can be considered as a kind of operator calculus on the manifold, similar to the semiclassical pseudodifferential calculus in the Euclidean space. Several basic notions and results of the semiclassical pseudodifferential calculus in the Euclidean space were extended to Berezin-Toeplitz operators on K\"ahler manifolds by Charles \cite{charles03}. 

In this paper, we are interested in asymptotic spectral properties of self-adjoint Toeplitz operators in semiclassical limit. Several aspects of the semiclassical spectral analysis were studied for Toeplitz operators on compact K\"ahler manifolds: for instance, quantum ergodicity \cite{Zelditch97}, the Gutzwiller trace formula \cite{BPU}, the Bohr-Sommerfeld conditions (both regular \cite{charles03a,charles06} and singular \cite{LeFloch1,LeFloch2}). We consider Toeplitz operators associated with the renormalized Bochner Laplacian on an arbitrary quantizable compact symplectic manifold. We assume that the principal symbol of a self-adjoint Toeplitz operator has a non-degenerate minimum with discrete wells and study the asymptotic behavior, in the semiclassical limit, of its eigenvalues at the bottom of the spectrum (low-lying eigenvalues) and of the corresponding eigensections. For semiclassical pseudodifferential operators (in particular, for the Schr\"odinger operator $-h^2\Delta+V$ in $\mathbb R^n$ with potential wells), the study of similar problems goes back to the papers of Helffer-Sj\"ostrand \cite{HelSj1}, Helffer-Robert \cite{HelRo} and Simon \cite{Simon}. For compact K\"ahler manifolds, such problems were recently studied by Deleporte in \cite{deleporte1,deleporte2}. For compact K\"ahler surfaces, a full asymptotic expansion for the first eigenvalues, valid on a fixed interval, was obtained in \cite{LeFloch1}.

As an immediate application of our results, we obtain asymptotic upper bounds, in the semiclassical limit, for low-lying eigenvalues of the Bochner Laplacian with discrete wells, thus extending to the case of arbitrary even dimension the results of \cite{HM01,Helffer-K11} on eigenvalue asymptotics for the two-dimen\-sional magnetic Schr\"odinger operator with non-vanishing magnetic field and discrete wells. 

\subsection{Preliminaries on Toeplitz operators}
 Let $(X,\omega)$ be a compact symplectic manifold of dimension $2n$ and $(L,h^L)$ be a Hermitian line bundle on $X$ with a Hermitian connection 
\[
\nabla^L : C^\infty(X,L)\to C^\infty(X,T^*X\otimes L).
\]
 The curvature of this connection is given by $R^L=(\nabla^L)^2$. 
We will assume that $L$ satisfies the pre-quantization condition:
\begin{equation}\label{e:def-omega}
\frac{i}{2\pi}R^L=\omega. 
\end{equation}
Thus, $[\omega]\in H^2(X,\mathbb Z)$. 

Let $g$ be a Riemannian metric on $X$. Let $J_0 : TX\to TX$ be a skew-adjoint operator such that 
\[
\omega(u,v)=g(J_0u,v), \quad u,v\in TX. 
\]
Consider the operator $J : TX\to TX$ given by 
\begin{equation}\label{e:defJ}
J=J_0(-J^2_0)^{-1/2}.
\end{equation} 
Then $J$ is an almost complex structure compatible with $\omega$ and $g$, that is, $g(Ju,Jv)=g(u,v), \omega(Ju,Jv)=\omega(u,v)$ and $\omega(u,Ju)\geq 0$ for any $u,v\in TX$.  
  
For any $p\in \NN$, let $L^p:=L^{\otimes p}$ be the $p$-th tensor power of $L$.  Let $\nabla^{L^p}: {C}^\infty(X,L^p)\to
{C}^\infty(X, T^*X \otimes L^p)$ be the connection
on $L^p$ induced by $\nabla^{L}$.
Denote by $\Delta^{L^p}$ the induced
Bochner Laplacian acting on $C^\infty(X,L^p)$ by
\begin{equation}\label{e:DeltaLp}
\Delta^{L^p}=\big(\nabla^{L^p}\big)^{\!*}\,
\nabla^{L^p},
\end{equation}
where $\big(\nabla^{L^p}\big)^{\!*}:
{C}^\infty(X,T^*X\otimes L^p)\to
{C}^\infty(X,L^p)$ denotes the formal adjoint of 
the operator $\nabla^{L^p}$.  If $\{e_j\}_{j=1,\ldots,2n}$ is a local orthonormal frame of $TX$, then $\Delta^{L^p}$ is given by
\begin{equation}\label{e:DeltaLp1}
\Delta^{L^p}=-\sum_{j=1}^{2n}\left[(\nabla^{L^p}_{e_j})^2-\nabla^{L^p}_{\nabla^{TX}_{e_j}e_j} \right],
\end{equation}
where $\nabla^{TX}$ is the Levi-Civita connection of the metric $g$. 

The renormalized Bochner Laplacian $\Delta_p$ is a second order differential operator acting on $C^\infty(X,L^p)$ by
 \[
\Delta_p=\Delta^{L^p}-p\tau,
 \]
where $\tau$ is a smooth function on $X$ given by 
 \begin{equation}\label{e:def-tau}
 \tau(x)=-\pi \operatorname{Tr}[J_0(x)J(x)],\quad x\in X.
 \end{equation}
This operator was introduced by Guillemin-Uribe in \cite{Gu-Uribe}. When $(X,\omega)$ is a K\"ahler manifold, it is twice the corresponding Kodaira Laplacian on functions $\Box^{L^p}=\bar\partial^{L^p*}\bar\partial^{L^p}$.
 
Denote by $\sigma(\Delta_p)$ the spectrum of $\Delta_p$ in $L^2(X,L^p)$. Put
  \begin{equation}\label{e:def-mu0}
 \mu_0=\inf_{u\in T_xX, x\in X}\frac{iR^L_x(u,J(x)u)}{|u|_g^2}.
 \end{equation}
By \cite[Cor. 1.2]{ma-ma02}, there exists a constant $C_L>0$ such that for any $p$
 \[
 \sigma(\Delta_p)\subset [-C_L,C_L]\cup [2p\mu_0-C_L,+\infty).
 \]
Consider the finite-dimensional vector subspace $\mathcal H_p\subset L^2(X,L^p)$ spanned by the eigensections of $\Delta_p$ corresponding to eigenvalues in $[-C_L,C_L]$. By \cite[Cor. 1.2]{ma-ma02} (also cf. \cite{B-Uribe,Gu-Uribe} in the case $J_0 = J$), its dimension is given, for $p$ large enough, by the Riemann-Roch-Hirzebruch formula
\[
d_p:=\dim \mathcal H_p=\langle \operatorname{ch}(L^p)\operatorname{Td}(TX), [X]\rangle.
\]
Here $\operatorname{ch}(L^p)$ is the Chern character of $L^p$ and $\operatorname{Td}(TX)$ is the Todd class of the tangent bundle $TX$ considered as a complex vector bundle with complex structure $J$. In particular, we have $d_p\sim p^n\int_X\frac{\omega^n}{n!}$ as $p\to \infty$.

Let $P_{\mathcal H_p}$ be the orthogonal projection from $L^2(X,L^p)$ onto $\mathcal H_p$. 

\begin{defn}
A Toeplitz operator is a sequence $\{T_p\}=\{T_p\}_{p\in \mathbb N}$ of bounded linear operators $T_p : L^2(X,L^p)\to L^2(X,L^p)$, satisfying the following conditions:
\begin{description}
\item[(i)] For any $p\in \mathbb N$, we have 
\[
T_p=P_{\mathcal H_p}T_pP_{\mathcal H_p}. 
\]
\item[(ii)] There exists a sequence $g_l\in C^\infty(X)$ such that 
\begin{equation}\label{e:Tii}
T_p=P_{\mathcal H_p}\left(\sum_{l=0}^\infty p^{-l}g_l\right)P_{\mathcal H_p}+\mathcal O(p^{-\infty}),
\end{equation}
i.e. for any natural $k$ there exists $C_k>0$ such that, for any $p\in \mathbb N$, 
\[
\left\|T_p-P_{\mathcal H_p}\left(\sum_{l=0}^k p^{-l}g_l\right)P_{\mathcal H_p}\right\|\leq C_kp^{-k-1}.
\]
\end{description}
  \end{defn}
 
Here we follow the approach to Toeplitz operator calculus introduced in \cite{ma-ma:book,ma-ma08} for spin$^c$ Dirac operator and K\"ahler case (also with an auxiliary bundle) and extended to the case under consideration in \cite{ioos-lu-ma-ma,bergman}.
 
The full symbol of $\{T_p\}$ is the formal series $\sum_{l=0}^\infty \hbar^lg_l\in C^\infty(X)[[\hbar ]]$ and the principal symbol of $\{T_p\}$ is $g_0$.  In the particular case when $g_l=0$ for $l\geq 1$ and $g_0=f$, we get the operator $T_{f,p}=P_{\mathcal H_p}f P_{\mathcal H_p}: L^2(X,L^p)\to L^2(X,L^p)$. 

If each $T_p$ is self-adjoint, then $\{T_p\}$ is called self-adjoint.
For real-valued $f$, the operator $T_{f,p}$ is obviously self-adjoint. Conversely, if $\{T_p\}$ is self-adjoint, then its principal symbol is real-valued. 

As shown in \cite{ioos-lu-ma-ma,bergman}, the set of Toeplitz operators forms an algebra. This is based on the Bergman kernel expansion from \cite{ma-ma08a} and Toeplitz operator calculus developed in \cite{ma-ma08} for spin$^c$ Dirac operator and K\"ahler case (also with an auxiliary bundle).

\begin{thm}\label{t:algebra}
Let $f,g\in C^\infty(X)$. Then the product of the Toeplitz operators $T_{f,p}$ and $T_{g,p}$ is a Toeplitz operator. More precisely, it admits the asymptotic expansion 
\begin{equation}\label{e:compT}
T_{f,p}T_{g,p}=\sum_{r=0}^\infty p^{-r}T_{C_r(f,g),p}+\mathcal O(p^{-\infty}), 
\end{equation}
with some $C_r(f,g)\in C^\infty(X)$, where $C_r$ are bidifferential operators. In particular, 
\[
C_0(f,g)=fg, \quad C_1(f,g)-C_1(g,f)=i\{f,g\},
\]
where $\{f,g\}$ is the Poisson bracket on $(X,2\pi \omega)$.
\end{thm}

Thus, the Toeplitz operators provide a Berezin-Toeplitz quantization for the compact symplectic manifold $(X,2\pi\omega)$. The limit $p\to +\infty$ for Toeplitz operators can be thought of as a semiclassical limit, with semiclassical parameter $\hbar=\frac{1}{p}\to 0$. Theorem~\ref{t:algebra} shows that this quantization has a correct semiclassical limit. 

Throughout in the paper, we will consider a self-adjoint Toeplitz operator $T_p$ with principal symbol $h$:
\begin{equation}\label{e:Tp}
T_p=P_{\mathcal H_p}\left(\sum_{l=0}^\infty p^{-l}g_l\right)P_{\mathcal H_p}+\mathcal O(p^{-\infty}), \quad g_0=h. 
\end{equation}
Without loss of generality, we will assume that the principal symbol $h$ satisfies the condition:
\begin{equation}\label{e:minh=0}
\min_{x\in X}h(x)=0.
\end{equation}

We will suppose that the operator $T_{p}$ acts on the finite-dimensional space $\cH_p$. Thus, the spectrum of $T_p$ consists of a finite number of eigenvalues
$\lambda^0_p\leq \lambda^1_p\leq \ldots \leq \lambda^{d_p-1}_p$. 
We will study the asymptotic properties of eigenvalues and eigensections of $T_p$ in the semiclassical limit $p\to \infty$. 

\subsection{Asymptotic decay of eigensections}
First,  we study localization properties of eigensections of Toeplitz operators. We start with a result on exponential decay of the eigensections of $T_p$ in the classically forbidden region. For any $h_0\geq 0$, denote 
\[
U_{h_0}=\{x\in X : h(x)\leq h_0\}.
\]

\begin{thm}\label{t:forbidden}
Let $\{T_p\}$ be a self-adjoint Toeplitz operator \eqref{e:Tp}, satisfying the condition \eqref{e:minh=0}. Assume that there exist $C>0$ and $a>0$ such that for any $p\in \NN$ and $(x,y)\in X\times X$, 
\begin{equation}\label{e:expT-h}
|(T_p-P_{\mathcal H_p}hP_{\mathcal H_p})(x,y)|<Cp^{-1} e^{-a\sqrt{p}d(x,y)},
\end{equation}
where $(T_p-P_{\mathcal H_p}hP_{\mathcal H_p})(x,y)$ denotes the Schwartz kernel of the operator $T_p-P_{\mathcal H_p}hP_{\mathcal H_p}$ with respect to the Riemannian volume form $dv_X$ on $X$ and $d$ is the geodesic distance function of $(X,g)$. 
 
Suppose that, for any $p\in \NN$. $\lambda_p$ is an eigenvalue of $T_{p}$ such that $\lambda_p \leq h_1$ with some $h_1>0$ independent of $p$ and $u_p\in \mathcal H_p$ is the corresponding normalized eigensection:
\[
T_{p}u_p=\lambda_pu_p, \quad \|u_p\|=1.
\]
Then, for any $h_0>h_1$, there exist $\alpha>0$ and $C_1>0$ such that, for any $p\in \NN$,
\[
\int_X e^{2\alpha\sqrt{p}d(x,U_{h_0})}|u_p(x)|^2 dv_X(x)<C_1, 
\]
where $d(x,U_{h_0})$ denotes the distance from $x$ to $U_{h_0}$.  
\end{thm}

The proof uses an approach to the study of Bergman kernels, developed in \cite{bergman,Ko-ma-ma}. It combines the methods of \cite{dai-liu-ma,ma-ma:book,ma-ma08a,ma-ma08} and weighted estimates with appropriate exponential weights as in \cite{Kor91}.  

Next, we consider the eigensections, corresponding to low-lying eigenvalues of a Toeplitz operator, under the assumption that the minimum of its principal symbol is non-degenerate. One can use the methods of the proof of Theorem~\ref{t:forbidden} in this setting and get some exponential decay estimates (see Theorem~\ref{low-l}). This result turns out to be insufficient for our applications. So instead we extend a general localization result proved by  Deleporte \cite{deleporte1} in the case when $X$ is a K\"ahler manifold.   

So let $\{T_p\}$ be a self-adjoint Toeplitz operator \eqref{e:Tp}, satisfying the condition \eqref{e:minh=0}. We assume that $h$ is non-degenerate on the zero set $U_0=\{x\in X : h(x)=0\}$: there exists $\alpha>0$ such that
\begin{equation}\label{e:nondeg0}
h(x)\geq \alpha d(x,U_0)^{2}, \quad x\in X.
\end{equation}
For any $\delta>0$, put 
\[
V_\delta=\{x\in X : d(x, U_0)<\delta\}.
\]

\begin{thm}\label{low-l0}
For any $c>0$, $k\in \NN$ and $\delta>0$, there exists $C>0$ such that for any $p\in \NN$ and $u_p\in \mathcal H_p$ such that  
\begin{equation}\label{e:Tp1}
T_{p}u_p=\lambda_pu_p, \quad \|u_p\|=1.
\end{equation}
with $\lambda_p$, satisfying the estimate 
\begin{equation}\label{e:lp1}
\lambda_p<\frac{c}{p}, 
\end{equation}
we have
\begin{equation}\label{e:minusVdelta}
\int_{X\setminus V_\delta} |u_p(x)|^2 dv_X(x)<Cp^{-k}, \quad p\in \NN.
\end{equation}
\end{thm}

\subsection{Asymptotic expansions of low-lying eigenvalues}\label{s:asymp-intro}
Let $T_p$ be a self-adjoint Toeplitz operator \eqref{e:Tp} with the principal symbol $h$, satisfying \eqref{e:minh=0}. For each non-degenerate minimum $x_0$ of $h$:
\[
h(x_0)=0, \quad {\rm Hess}\,h(x_0)>0,
\]
one can define the model operator for $\{T_{p}\}$ at $x_0$ in the following way. 

The second order term in the Taylor expansion of $h$ at $x_0$ (in normal coordinates near $x_0$): 
\[
q_{x_0}(Z)=\left(\frac12 {\rm Hess}\,h(x_0)Z,Z\right)
\]
is a positive quadratic form on $T_{x_0}X\cong \mathbb R^{2n}$.  

Consider a second order differential operator $\mathcal L_{x_0}$ in $C^\infty(T_{x_0}X)$ given by
\begin{equation}\label{e:defL0}
\mathcal L_{x_0}=-\sum_{j=1}^{2n} \left(\nabla_{e_j}+\frac 12 R^L_{x_0}(Z,e_j)\right)^2-\tau (x_0), 
\end{equation}
where $\{e_j\}_{j=1,\ldots,2n}$ is an orthonormal base in $T_{x_0}X$.  Here, for $U\in T_{x_0}X$, we denote by $\nabla_U$ the ordinary operator of differentiation in the direction $U$ on $C^\infty(T_{x_0}X)$.  Let $\mathcal P_{x_0}$ be the orthogonal projection in $L^2(T_{x_0}X)$ to the kernel of $\mathcal L_{x_0}$ (see Section~\ref{s:char} for more details).

The model operator for $\{T_{p}\}$ at $x_0$ is the Toeplitz operator $\mathcal T_{x_0}$ in $L^2(T_{x_0}X)$ given by
\[
\mathcal T_{x_0}=\mathcal P_{x_0}(q_{x_0}(Z)+g_1(x_0))\mathcal P_{x_0},
\]
where $g_1$ is the second coefficient in the asymptotic expansion \eqref{e:Tp} for the operator $\{T_{p}\}$. It is an unbounded self-adjoint operator in $L^2(T_{x_0}X)$ with discrete spectrum (see also below). The eigenvalues of $\mathcal T_{x_0}$ do not depend on the choice of normal coordinates, and the lowest eigenvalue is simple. 

Now suppose that each minimum of $h$ is non-degenerate. Then the zero set $U_0$ is a finite set of points:
\[
U_0=\{x_1,\ldots, x_N\}.
\]
Let $\mathcal T$ be the self-adjoint operator on $L^2(T_{x_1}X)\oplus \ldots \oplus  L^2(T_{x_N}X)$ defined by 
\begin{equation}\label{e:def-Tau}
\mathcal T=\mathcal T_{x_1}\oplus \ldots \oplus \mathcal T_{x_N}.
\end{equation}

\begin{thm}\label{t:DelB}
Let $\{\lambda^m_p\}$ be the non-decreasing sequence of the eigenvalues of $T_{p}$ on $\mathcal H_p$ (counted with multiplicities) and let $\{\mu_m\}$ be  
 the non-decreasing sequence of the eigenvalues of $\mathcal T$ (counted with multiplicities). 
Then, for any fixed $m$, $\lambda^m_p$ has an asymptotic expansion, when $p\to \infty$, of the form
\begin{equation}\label{e:DelB}
\lambda^m_p=p^{-1}\mu_m+p^{-3/2}\phi_m+\mathcal O(p^{-2})
\end{equation}
with some $\phi_m\in \mathbb R$. 
\end{thm}

\begin{thm}\label{t:DelA}
If $\mu$ is a simple eigenvalue of $\mathcal T_{x_j}$ for some $j=1,\ldots,N$, then there exists a sequence $\lambda_p$ of eigenvalues of $T_{p}$ on $\mathcal H_p$ which admits a complete asymptotic expansion of the form 
\[
\lambda_p\sim p^{-1}\sum_{k=0}^{+\infty}a_kp^{-k}, \quad a_0=\mu.
\]
\end{thm}

In the case when $(X,\omega)$ is a K\"ahler manifold, that is, $J=J_0$ and $J$ is a complex structure, these results are slight refinements of the main results of \cite{deleporte1}.   

\subsection{Applications to the Bochner Laplacian} 
We assume that the function $\tau\in C^\infty(X)$ satisfies the following conditions:
\begin{itemize}
\item $\min_{x\in X}\tau(x)=\tau_0$;
\item There exists a unique $x_0\in X$ such that $\tau(x)=\tau_0$, which is non-degenerate:
\[
{\rm Hess}\,\tau(x_0)>0. 
\]
\end{itemize}

Denote by 
\[
q_{x_0}(Z)=\left(\frac12 {\rm Hess}\,\tau(x_0)Z,Z\right), \quad Z\in T_{x_0}X.
\]
We will also need a smooth function $J_{1,2}$ on $X$, which appears in the leading coefficient in the asymptotic expansion of the restriction to the diagonal of the generalized Bergman kernel $P_{1,p}$ of $\Delta_p$, which is the smooth kernel of the operator $\Delta_pP_{\mathcal H_p}$ with respect to $dv_X$. We refer the reader to Section~\ref{s:41} for more information on this function. 

Consider the Toeplitz operator $\mathcal T_{x_0}$ in $L^2(T_{x_0}X)$ defined by
\[
\mathcal T_{x_0}=\mathcal P_{x_0}(q_{x_0}(Z)+J_{1,2}(x_0))\mathcal P_{x_0}.
\]
Denote by $\{\mu_j\}$ the non-decreasing sequence of the eigenvalues of $\mathcal T_{x_0}$ (counted with multiplicities). 

Let $\{\lambda_j(\Delta^{L^p})\}$ be the non-decreasing sequence of the eigenvalues of the operator $\Delta^{L^p}$ (counted with multiplicities).
As a consequence of Theorem \ref{t:DelB} and variational technique, we obtain upper estimates for $\lambda_j(\Delta^{L^p})$.

\begin{thm}\label{t:upper-estimates}
For any $j\in \NN$, there exists  $\phi_j\in \RR$ such that
\[
\lambda_j(\Delta^{L^p})\leq p\tau_0+\mu_j+p^{-1/2}\phi_j+\mathcal O(p^{-1}), \quad p\to\infty.
\]
\end{thm}

In the case when the line bundle $L$ is trivial, the corresponding Bochner Laplacian $\Delta^{L^p}$ is closely related with the magnetic Schr\"odinger operator $H^h$ with the semiclassical parameter $h=p^{-1}$
(see Section~\ref{s:4} for more details). The case when the form $\omega$ is symplectic corresponds to the case when the magnetic field is non-vanishing (of full rank). In the two-dimensional case, spectral properties of the magnetic Schr\"odinger operator with non-vanishing magnetic field were studied in \cite{Helffer-K11,HM01,Helffer-K12,Helffer-K15,RV15} (see also the survey paper \cite{Helffer-K14} and the book \cite{Raymond-book}). As shown in Section~\ref{s:2Dcase}, for a two-dimensional magnetic Laplacian, the upper bounds of Theorem~\ref{t:upper-estimates} are sharp and agree with the asymptotic expansions of \cite{Helffer-K11,HM01}.

The paper is organized as follows. In Section~\ref{s:2}, we study localization properties of eigensections of Toeplitz operators and prove Theorem~\ref{t:forbidden} and \ref{low-l0}. Section~\ref{s:3} is devoted to asymptotic expansions of low-lying eigenvalues. Here we prove Theorems~\ref{t:DelB} and \ref{t:DelA}. Finally, Section~\ref{s:4} contains applications to the Bochner Laplacian.

\section{Asymptotic decay of eigensections}\label{s:2}

\subsection{Differential operators and Sobolev spaces}
We will need to consider families of differential operators $A_p$ acting in $C^\infty(X,L^p)$ with $p\in \NN$. In this section, we collect some necessary results (see, for instance, \cite{Ko-ma-ma}). 

As usual, we introduce the $L^2$-norm on $L^2(X,L^p)$ by 
\[
\|u\|^2_{p,0}=\int_{X}|u(x)|^2dv_{X}(x), \quad u\in L^2(X,L^p).
\]
For any integer $m>0$, we introduce the norm $\|\cdot\|_{p,m}$ on the Sobolev space $H^m(X,L^p)$ of order $m$ by the formula
\[
\|u\|^2_{p,m}=\sum_{\ell=0}^m \int_{X} \left|\left(\frac{1}{\sqrt{p}}\nabla^{L^p}\right)^\ell u(x)\right|^2 dv_{X}(x), \quad u\in H^m(X,L^p).
\]
Denote by $\langle\cdot,\cdot\rangle_{p,m}$ the corresponding inner product on $H^m(X,L^p)$. For any integer $m<0$, we define the norm in the Sobolev space $H^m(X,L^p)$ by duality.  
For any bounded linear operator $A : H^m(X,L^p)\to H^{m^\prime}(X,L^p)$, $m,m^\prime\in \mathbb Z$, we will denote its norm by $\|A\|^{m,m^\prime}_p$.

Let $E$ be a vector bundle over $X$. Suppose that $E$ 
is Euclidean or Hermitian depending on whether it is real or complex 
and equipped with a metric connection $\nabla^E$. 
The Levi-Civita connection $\nabla^{TX}$ on $(X,g^{TX})$ and 
the connection $\nabla^E$ define a metric connection
$\nabla^E : {C}^\infty(X, (T^*X)^{\otimes j}
\otimes E)\to {C}^\infty(X, (T^*X)^{\otimes (j+1)}
\otimes E)$
on each vector bundle $(T^*X)^{\otimes j} \otimes E$ for 
$j\in \mathbb N$, that allows us to introduce the operator 
$$\big(\nabla^E\big)^{\!\ell} : 
{C}^\infty(X, E)
\to {C}^\infty(X, (T^*X)^{\otimes \ell} \otimes E)$$ 
for every $\ell\in \mathbb N$.
Any differential operator $A$ of order $q$ acting in 
${C}^\infty(X,E)$ can be written as
\begin{equation}\label{e2.1}
A=\sum_{\ell=0}^q a_\ell\cdot \big(\nabla^E\big)^{\!\ell},
\end{equation}
where $a_\ell \in {C}^\infty(X, (TX)^{\otimes \ell})$ 
and the endomorphism 
$\cdot : (TX)^{\otimes \ell}\otimes ((T^*X)^{\otimes \ell}
\otimes E) \to E$ is given by the contraction. 

Denote by $D^q(X,L^p)$ the space of differential operators of order $q$, acting on $C^\infty(X,L^p)$. We say that a family $\{A_p\in D^q(X,L^p), p\in \mathbb N\}$ is bounded in $p$, if 
\[
A_p=\sum_{\ell=0}^q a_{p,\ell}\cdot \left(\frac{1}{\sqrt{p}}\nabla^{L^p}\right)^\ell,\quad a_{p,\ell} \in C^\infty(X, (TX)^{\otimes \ell}),
\]
and, for any $\ell=0,1,\ldots,q$, the family $\{a_{p,\ell}, p\in \mathbb N\}$ is bounded in the Frechet space $C^\infty(X, (TX)^{\otimes \ell})$.
An example of a bounded in $p$ family of differential operators is given by $\{\frac{1}{p}\Delta_p, , p\in \mathbb N\}$.

Recall that any operator $A\in D^q(X,L^p)$ defines a bounded operator from $H^{m+q}(X,L^p)$ to $H^{m}(X,L^p)$ for any $m\in \mathbb Z$. The following proposition provides a control on the operator norms of a family of differential operators as above. 
 
\begin{prop}\label{p:Sobolev-mapping}
If a family $\{A_p\in D^q(X,L^p), p\in \mathbb N\}$ is bounded in $p$, then for any $m\in \mathbb Z$, there exists $C_m>0$ such that, for all $p\in \mathbb N$,  
\[
\|A_pu\|_{p,m}\leq C_m\|u\|_{p,m+q},\quad u\in H^{m+q}(X,L^p).
\]
\end{prop}

\subsection{Weighted estimates for eigensections}
Here we prove weighted estimates for eigensections of Toeplitz operators, satisfying \eqref{e:minh=0} and \eqref{e:expT-h}.  

Let $\rho$ be a Lipschitz function on $X$, that is, there exists $L>0$ such that 
\begin{equation}\label{e:Lip}
|\rho(x)-\rho(y)|\leq Ld(x,y), \quad x,y\in X.
\end{equation}

\begin{prop}\label{p:weighted-est}
Assume that $\{T_p\}$ is a self-adjoint Toeplitz operator \eqref{e:Tp}, satisfying the conditions \eqref{e:minh=0} and \eqref{e:expT-h}. Let $u_p$ be a sequence of  eigensections of $\{T_{p}\}$:
\begin{equation}\label{e:eigenT}
T_{p}u_p=\lambda_pu_p, \quad u_p\in {\cH_p}.
\end{equation}
There exist $C_1,C_2, C_3>0$ and $\alpha_0>0$ such that for any $\alpha\in \mathbb R$ such that $|\alpha|<\alpha_0$ and $p\in \NN$, 
\[
\int_X e^{2\alpha\sqrt{p} \rho(x)} \left[ h(x)-\left(1+C_1|\alpha|\right)\lambda_p-\frac{C_2|\alpha|}{\sqrt{p}}-\frac{C_3}{p}\right]|u_p(x)|^2\,dv_X(x)\leq 0.
\]
\end{prop}

The rest of this subsection is devoted to the proof of Proposition~\ref{p:weighted-est}. 
By a standard averaging procedure \cite[Proposition 4,1]{Kor91} (see also \cite[Section 3.1]{Ko-ma-ma}), one can construct a sequence $\tilde \rho_p\in C^\infty(X), p\in \NN,$ satisfying the following conditions:

(1) there exists $c_0>0$ such that 
\begin{equation}
\label{(1.1)}
\vert \tilde \rho_p(x) - \rho(x)\vert  < \frac{c_0}{\sqrt{p}},\quad x\in X, p\in \NN;
\end{equation}

(2) for any $k>0$, there exists $c_k>0$ such that  
\begin{equation}
\label{dist}
\left\| \tilde \rho_p\right\|_{C^k(X)} < c_{k}\left(\frac{1}{\sqrt{p}}\right)^{1-k},\quad p\in \NN.
\end{equation}

For any $\alpha \in \mathbb R$ and $p\in \NN$, consider the function $f_{\alpha,p}\in C^\infty(X)$ given by 
\[
f_{\alpha,p}(x)=e^{\alpha\sqrt{p}\tilde \rho_p(x)}, \quad x\in X. 
\]
These functions play role of weighted functions in our estimates. We will use the same notation $f_{\alpha,p}$ for the multiplication operator by $f_{\alpha,p}$ in $C^\infty(X,L^p)$.  Instead of working directly with the associated weighted spaces, we will consider operators of the form $\{f_{\alpha,p}Af_{-\alpha,p}\}$ and only at the very end switch to weighted estimates.  

Observe that, for $v\in C^\infty(X,TX)$, 
\[
\nabla^{L^p}_{\alpha;v}:=f_{\alpha,p}\nabla^{L^p}_{v}f^{-1}_{\alpha,p}=\nabla^{L^p}_{v}-\alpha\sqrt{p} d\tilde{\rho}_{p}(v).
\]
Therefore, for any $a>0$ and $v\in C^\infty(X,TX)$, the family $\{\frac{1}{\sqrt{p}}\nabla^{L^p}_{\alpha;v} : |\alpha|<a\}$ is a family of operators from $D^1(X,L^p)$, uniformly bounded in $p$.
The operator  $\Delta_{p;\alpha}:=f_{\alpha,p}\Delta_{p}f^{-1}_{\alpha,p}$ has the form
\begin{equation}\label{Lta}
\Delta_{p;\alpha}=\Delta_{p}+\alpha A_{p}+\alpha^2B_{p}, 
\end{equation}
where $A_{p}\in D^1(X,L^p)$ and $B_{p}\in D^0(X,L^p)$. Moreover, the families $\{\frac{1}{p}A_{p} : p\in \mathbb N\}$ and $\{\frac{1}{p}B_{p} : p\in \mathbb N\}$ are uniformly bounded in $p$.
Indeed, using \eqref{e:DeltaLp1}, one can check that, if $\{e_j\}_{j=1,\ldots,2n}$ is a local orthonormal frame of $TX$, then 
\begin{align}\label{e:Ap}
A_p=& \sqrt{p}  \sum_{j=1}^{2n}\left[\nabla^{L^p}_{e_j} \circ d\tilde{\rho}_{p}(e_j)+d\tilde{\rho}_{p}(e_j)\circ \nabla^{L^p}_{e_j}+d\tilde{\rho}_{p}(\nabla^{TX}_{e_j}e_j)  \right],\\ \label{e:Bp}
B_p=& -p\sum_{j=1}^{2n} d\tilde{\rho}_{p}(e_j)^2=-p\|d\tilde{\rho}_{p}\|^2.
\end{align}

\begin{lem}\label{l:1}
There exist $\alpha_1>0$ and $C>0$ such that, for any $p\in \NN$ and $|\alpha|<\alpha_1$, 
\[
\|f_{-\alpha,p}[P_{\mathcal H_p},f_{2\alpha,p}]f_{-\alpha,p}\|<C|\alpha|. 
\]
\end{lem}

\begin{proof}
Observe that 
\begin{equation}\label{e:1}
f_{-\alpha,p}[P_{\mathcal H_p},f_{2\alpha,p}]f_{-\alpha,p}=f^{-1}_{\alpha,p}P_{\mathcal H_p}f_{\alpha,p}-f_{\alpha,p}P_{\mathcal H_p}f^{-1}_{\alpha,p}. 
\end{equation}
Let $\delta$ be the counterclockwise oriented circle in $\mathbb C$ centered at $0$ of radius $\mu_0$. We will use the formula
\begin{equation}\label{bergman-integral}
P_{\mathcal H_p}=\frac{1}{2\pi i}
\int_\delta \left(\lambda-\frac{1}{p}\Delta_p\right)^{-1}d\lambda. 
\end{equation}
By \cite[Theorem 2.3]{Ko-ma-ma} (see also \cite[Theorem 1.7]{ma-ma08}), we have the following theorem. 

\begin{thm}\label{Thm1.7}
For any $\lambda\in \delta$ and $p\in \mathbb N$, the operator $\lambda-\frac{1}{p}\Delta_p$ is invertible in $L^2(X,L^p)$, and there exists $C>0$ such that for all $\lambda\in \delta$ and $p\in \mathbb N$, 
\[
\left\|\left(\lambda-\frac{1}{p}\Delta_p\right)^{-1}\right\|^{0,0}_p\leq C, \quad 
\left\|\left(\lambda-\frac{1}{p}\Delta_p\right)^{-1}\right\|^{0,1}_p\leq C.
\]
\end{thm} 

As in \cite[Theorem 3.2]{Ko-ma-ma} (see also \cite{bergman}), we can derive the corresponding weighted estimates.

\begin{thm}\label{Thm1.7W}
There exist $\alpha_1>0$ and $C>0$ such that, for all $\lambda\in \delta$, $p\in \mathbb N$, $|\alpha|<\alpha_1$, the operator $\lambda-\frac{1}{p}\Delta_{p;\alpha}$ is invertible in $L^2(X,L^p)$, and, for the inverse operator $(\lambda-\frac{1}{p}\Delta_{p;\alpha})^{-1}$, we have
\begin{equation}\label{e:weighted-res}
\left\|(\lambda-\frac{1}{p}\Delta_{p;\alpha})^{-1}\right\|^{0,0}_p\leq C, \quad
 \left\|(\lambda-\frac{1}{p}\Delta_{p;\alpha})^{-1}\right\|^{0,1}_p\leq C.
\end{equation}
\end{thm}

It is easy to see that the inverse operators $\left(\lambda-\frac{1}{p}\Delta_{p;\alpha}\right)^{-1}$ and $\left(\lambda-\frac{1}{p}\Delta_p\right)^{-1}$ are related by the formula
\begin{equation}\label{e:res-relation}
\left(\lambda-\frac{1}{p}\Delta_{p;\alpha}\right)^{-1}=f_{\alpha,p}\left(\lambda-\frac{1}{p}\Delta_p\right)^{-1}f^{-1}_{\alpha,p}.
\end{equation}
By \eqref{bergman-integral}, it follows that
\begin{equation}\label{e:P-weight}
f_{\alpha,p}P_{\mathcal H_p}f^{-1}_{\alpha,p}=\frac{1}{2\pi i}
\int_\delta f_{\alpha,p}\left(\lambda-\frac{1}{p}\Delta_p\right)^{-1}f^{-1}_{\alpha,p}d\lambda. 
\end{equation}
Using the resolvent identity and \eqref{Lta}, we get:
\begin{multline}\label{e:est}
f_{\alpha,p}\left(\lambda-\frac{1}{p}\Delta_p\right)^{-1}f^{-1}_{\alpha,p}-f^{-1}_{\alpha,p}\left(\lambda-\frac{1}{p}\Delta_p\right)^{-1}f_{\alpha,p}\\=\left(\lambda-\frac{1}{p}\Delta_{p,\alpha}\right)^{-1}\left(\frac{1}{p}\Delta_{p,\alpha}-\frac{1}{p}\Delta_{p,-\alpha}\right)\left(\lambda-\frac{1}{p}\Delta_{p,-\alpha}\right)^{-1}\\ =\frac{2\alpha}{p} \left(\lambda-\frac{1}{p}\Delta_{p,\alpha}\right)^{-1}A_{p}\left(\lambda-\frac{1}{p}\Delta_{p,-\alpha}\right)^{-1}.
\end{multline}
Using \eqref{e:weighted-res} and \eqref{e:est}, we proceed as follows:
\begin{multline*}
\left\|f_{\alpha,p}\left(\lambda-\frac{1}{p}\Delta_p\right)^{-1}f^{-1}_{\alpha,p}-f^{-1}_{\alpha,p}\left(\lambda-\frac{1}{p}\Delta_p\right)^{-1}f_{\alpha,p} \right\|^{0,0}_p \\ \leq \frac{2|\alpha|}{p} \left\| \left(\lambda-\frac{1}{p}\Delta_{p,\alpha}\right)^{-1}\right\|^{0,0}_p \left\|A_{p}\left(\lambda-\frac{1}{p}\Delta_{p,-\alpha}\right)^{-1}\right\|^{0,0}_p\\ \leq  C|\alpha|\left\| \frac{1}{p} A_{p}\right\|^{1,0}_p\leq  C_1|\alpha|.
\end{multline*}
Taking into account \eqref{e:1} and \eqref{e:P-weight}, this completes the proof. 
\end{proof}

\begin{cor}\label{c:1}
There exists $C>0$ such that, for any $p\in \NN$ and $|\alpha|<\alpha_1$, 
\[
\|f_{\alpha,p}P_{\mathcal H_p}f_{-\alpha,p}\|<C. 
\]
\end{cor}

\begin{lem}\label{l:2}
There exists $C>0$ such that, for any $p\in \NN$ and $\alpha$ such that $|\alpha|<\alpha_1$ with $\alpha_1$ as in Lemma~\ref{l:1}, we have
\[
\|f_{\alpha,p} [h,P_{\mathcal H_p}] f^{-1}_{\alpha,p}\|\leq \frac{C}{\sqrt{p}}. 
\]
\end{lem}

\begin{proof}
By \eqref{e:res-relation} and \eqref{e:P-weight}, we can write 
\begin{multline}\label{e:271}
f_{\alpha,p} [h,P_{\mathcal H_p}] f^{-1}_{\alpha,p}= [h,f_{\alpha,p}P_{\mathcal H_p}f^{-1}_{\alpha,p}]\\ =\frac{1}{2\pi i}
\int_\delta \left[h, \left(\lambda-\frac{1}{p}\Delta_{p,\alpha}\right)^{-1}\right] d\lambda. 
\end{multline}
We clearly have
\begin{multline}\label{e:272}
 \left[h, \left(\lambda-\frac{1}{p}\Delta_{p,\alpha}\right)^{-1}\right]\\ =\left(\lambda-\frac{1}{p}\Delta_{p,\alpha}\right)^{-1}\left[h, \frac{1}{p}\Delta_{p,\alpha}\right] \left(\lambda-\frac{1}{p}\Delta_{p,\alpha}\right)^{-1}.
\end{multline}
We use the formula \eqref{Lta} to compute the commutator $\left[h, \frac{1}{p}\Delta_{p,\alpha}\right]$:
\[
\left[h, \frac{1}{p}\Delta_{p,\alpha}\right]=\left[h, \frac{1}{p} \Delta_{p}\right]+\alpha\left[h, \frac{1}{p}A_{p}\right].  
\]
Using the formulas \eqref{e:DeltaLp1} and \eqref{e:Ap}, one can compute that, if $\{e_j\}_{j=1,\ldots,2n}$ is a local orthonormal frame of $TX$, then 
\begin{align}
\label{e:DeltaLph}
\left[h, \frac{1}{p} \Delta_{p}\right]=& \frac{1}{p}\sum_{j=1}^{2n}\left[dh(e_j) \circ \nabla^{L^p}_{e_j}+ \nabla^{L^p}_{e_j}\circ dh(e_j) -dh(\nabla^{TX}_{e_j}e_j) \right]\\
\label{e:Aph}
\left[h, \frac{1}{p}A_{p}\right] =& -\frac{2}{\sqrt{p}}  \sum_{j=1}^{2n} dh(e_j)\cdot  d\tilde{\rho}_{p}(e_j)=-\frac{2}{\sqrt{p}}  (dh,  d\tilde{\rho}_{p})_{g^{T^*X}}.
\end{align}
By \eqref{e:DeltaLph} and \eqref{e:Aph}, we see that, for any $\alpha$, the family $\sqrt{p}\left[h, \frac{1}{p}\Delta_{p,\alpha}\right]$ is a bounded family in $D^1(X,L^p)$. By Proposition \ref{p:Sobolev-mapping}, it follows that 
\[
\left\| \left[h, \frac{1}{p}\Delta_{p,\alpha}\right]\right\|^{1,0}_p\leq  C_1\frac{1}{\sqrt{p}}.
\]
By this estimate, \eqref{e:271}, \eqref{e:272} and Theorem~\ref{Thm1.7W}, we immediately complete the proof. 
\end{proof}

Now we are ready to prove Proposition~\ref{p:weighted-est}. 

\begin{proof}[Proof of Proposition~\ref{p:weighted-est}]
Using \eqref{e:eigenT}, we get
\begin{multline}\label{e:p2.2}
(f_{\alpha,p}hu_p, f_{\alpha,p}u_p)\\
\begin{aligned}
=& (P_{\mathcal H_p}f_{2\alpha,p}hu_p,u_p)\\
= &(f_{\alpha,p}P_{\mathcal H_p}hu_p,f_{\alpha,p}u_p)+([P_{\mathcal H_p},f_{2\alpha,p}]hu_p,u_p)\\
= &\lambda_p\|f_{\alpha,p}u_p\|^2- (f_{\alpha,p}(T_p-P_{\mathcal H_p}hP_{\mathcal H_p})u_p,f_{\alpha,p}u_p)
\end{aligned}
\\ +([P_{\mathcal H_p},f_{2\alpha,p}]hu_p,u_p).
\end{multline}
The kernel $K_{\alpha,p}\in C^\infty(X\times X)$ of the operator $f_{\alpha,p}(T_p-P_{\mathcal H_p}hP_{\mathcal H_p})f_{-\alpha,p}$ is given by 
\[
K_{\alpha,p}(x,y)=e^{\alpha\sqrt{p}\tilde \rho_p(x)}(T_p-P_{\mathcal H_p}hP_{\mathcal H_p})(x,y)e^{-\alpha\sqrt{p}\tilde \rho_p(y)}.
\]
Using \eqref{(1.1)}, \eqref{e:Lip} and \eqref{e:expT-h}, one can show that there exists $C>0$ such that, for any $p\in \NN$ and $(x,y)\in X\times X$, 
\[
|K_{\alpha,p}(x,y)|\leq Cp^{-1}e^{(L|\alpha|-a)\sqrt{p}d(x,y)}.
\]
Therefore, if we take $|\alpha|<a/L$, then, for any $p\in \NN$ and $(x,y)\in X\times X$, 
\[
|K_{\alpha,p}(x,y)|\leq Cp^{-1},
\]
and, by Schur's lemma, there exists $C_3>0$ such that, for any $p\in \NN$, 
\begin{equation}\label{e:T-PhP}
\|f_{\alpha,p}(T_p-P_{\mathcal H_p}hP_{\mathcal H_p})f_{-\alpha,p}\|\leq C_3p^{-1}, \quad p\in \NN.
\end{equation}
By \eqref{e:T-PhP}, we get the bound for the second term in the right-hand side of \eqref{e:p2.2}: 
\[
|(f_{\alpha,p}(T_p-P_{\mathcal H_p}hP_{\mathcal H_p})u_p,f_{\alpha,p}u_p)|\leq \frac{C_3}{p}\|f_{\alpha,p}u_p\|^2,  \quad p\in \NN.
\]

Put $\alpha_0=\min(a/L,\alpha_1)$, where $\alpha_1$ is given by Lemma~\ref{l:1}, and assume that $|\alpha|<\alpha_0$.
For the third term in the right-hand side of \eqref{e:p2.2}, we proceed as follows:
\begin{multline*}
([P_{\mathcal H_p},f_{2\alpha,p}]hu_p,u_p)\\
\begin{aligned}
=&([P_{\mathcal H_p},f_{2\alpha,p}] P_{\mathcal H_p} hu_p,u_p)+([P_{\mathcal H_p},f_{2\alpha,p}](I-P_{\mathcal H_p})hu_p,u_p)\\
=&\lambda_p([P_{\mathcal H_p},f_{2\alpha,p}]u_p,u_p)+([P_{\mathcal H_p},f_{2\alpha,p}] (T_p-P_{\mathcal H_p} hP_{\mathcal H_p})u_p,u_p)\\&+([P_{\mathcal H_p},f_{2\alpha,p}][h,P_{\mathcal H_p}] u_p,u_p).
\end{aligned}
\end{multline*}
Using Lemma~\ref{l:1}, Lemma \ref{l:2} and \eqref{e:T-PhP}, for any $p\in \NN$, we get the estimate 
\[
|([P_{\mathcal H_p},f_{2\alpha,p}]hu_p,u_p)|
\leq C|\alpha|\lambda_p \|f_{\alpha,p} u_p\|^2+\frac{C_2|\alpha|}{\sqrt{p}}\|f_{\alpha,p}u_p\|^2,
\]
that immediately completes the proof. 
\end{proof}

\subsection{Tunneling estimates}\label{s:tunnel}
Here we prove Theorem~\ref{t:forbidden}, the result on exponential decay of the eigensections in the classically forbidden region.

Let $\rho(x)=d(x,U_{h_0})$. This is a Lipschitz function, so we can construct a sequence $\tilde \rho_p\in C^\infty(X), p\in \NN,$ satisfying \eqref{(1.1)} and \eqref{dist}. By Proposition~\ref{p:weighted-est}, it follows that, for any $M\in (h_1, h_0)$, there exists $\alpha>0$ such that for any $p\in \NN$,
\begin{equation}\label{e:ineq00}
\int_X e^{2\alpha\sqrt{p}\tilde \rho_p(x)}\left(h(x)-M\right)|u_p(x)|^2 dv_X(x)\leq 0.
\end{equation}
Now the proof is completed as in \cite{HelSj1}. For convenience of the reader, we recall the arguments.

We rewrite \eqref{e:ineq00} in the form
\begin{multline}\label{e:ineq}
\int_{X\setminus U_{h_0}} e^{2\alpha\sqrt{p}\tilde \rho_p(x)}\left(h(x)-M\right)|u_p(x)|^2 dv_X(x)\\ \leq \int_{U_{h_0}} e^{2\alpha\sqrt{p}\tilde \rho_p(x)}\left(M-h(x)\right)|u_p(x)|^2 dv_X(x).
\end{multline}
For the left-hand side of \eqref{e:ineq},  since $h(x)\geq h_0$ on $X\setminus U_{h_0}$, we get
\begin{multline}\label{e:ineq31}
\int_{X\setminus U_{h_0}} e^{2\alpha\sqrt{p}\tilde \rho_p(x)}\left(h(x)-M\right)|u_p(x)|^2 dv_X(x)\\ \geq (h_0-M) \int_{X\setminus U_{h_0}} e^{2\alpha\sqrt{p}\tilde \rho_p(x)}|u_p(x)|^2 dv_X(x).
\end{multline}
If $x\in U_{h_0}$, then $\rho(x)=0$ and, by \eqref{(1.1)}, $\tilde{\rho}_p(x)<c_0/\sqrt{p}$. Hence, we have
\begin{multline}\label{e:ineq30}
\int_{U_{h_0}} e^{2\alpha\sqrt{p}\tilde \rho_p(x)}|u_p(x)|^2 dv_X(x)\\ \leq e^{2c_0\alpha} \int_{U_{h_0}}|u_p(x)|^2 dv_X(x)\leq e^{2c_0\alpha} \int_{X}|u_p(x)|^2 dv_X(x).
\end{multline}
Using \eqref{e:ineq30} and the fact that $h$ is bounded, for the right-hand side of \eqref{e:ineq}, we get the estimate
\begin{multline}\label{e:ineq1}
\int_{U_{h_0}} e^{2\alpha\sqrt{p}\tilde \rho_p(x)}\left(M-h(x)\right)|u_p(x)|^2 dv_X(x)\\ \leq C \int_{U_{h_0}} e^{2\alpha\sqrt{p}\tilde \rho_p(x)}|u_p(x)|^2 dv_X(x)  \leq C_1 \int_{X}|u_p(x)|^2 dv_X(x).
\end{multline}
By \eqref{e:ineq31}, \eqref{e:ineq} and \eqref{e:ineq1}, it follows that
\begin{equation}\label{e:ineq2}
\int_{X\setminus U_{h_0}} e^{2\alpha\sqrt{p}\tilde \rho_p(x)}|u_p(x)|^2 dv_X(x) \leq C_2 \int_X|u_p(x)|^2 dv_X(x).
\end{equation}
Finally, by \eqref{e:ineq30} and \eqref{e:ineq2}, we infer that, for $p>p_0$, 
\[
\int_X e^{2\alpha\sqrt{p}\tilde \rho_p(x)}|u_p(x)|^2 dv_X(x) \leq C_2 \int_X|u_p(x)|^2 dv_X(x),
\]
 that completes the proof of Theorem~\ref{t:forbidden}.

\subsection{Low-lying eigenvalues} In this section, we study localization properties of eigensections, corresponding to low-lying eigenvalues, under the assumption that the minimum of the principal symbol is non-degenerate. 

First, we prove Theorem~\ref{low-l0}. 

\begin{proof}[Proof of Theorem~\ref{low-l0}]
The proof is obtained by a slight modification of the proof of \cite[Proposition 3.1]{deleporte1}. So we will be sketchy. 

First, we show that, for any $k\in \NN$, the operator $T_p^k$ has the form
\begin{equation}\label{e:Tpk}
T_p^k=P_{\cH_p} \left(h^k+\sum_{j=1}^{k-1}p^{-j}f_{j,k}+\cO(p^{-k})\right) P_{\cH_p},
\end{equation}
where, for any $k$ and $j=1,\ldots,k-1$, $f_{j,k}\in C^\infty(X)$ and there exists $C_{jk}>0$ such that 
\begin{equation}\label{e:fjk}
|f_{j,k}(x)|<C_{jk}h(x)^{k-j},\quad  x\in X. 
\end{equation}

Iterating \eqref{e:compT}, one can see that, for any $f_1,\ldots,f_k\in C^\infty(X)$, the product of the Toeplitz operators $T_{f_1,p} \ldots T_{f_k,p}$ is a Toeplitz operator, which admits the asymptotic expansion 
\begin{equation}\label{e:compTk}
T_{f_1,p} \ldots T_{f_k,p}\\ =P_{\cH_p} \left(\sum_{r=0}^{k-1}p^{-r}C_{r,k}(f_1,\ldots,f_k)+\cO(p^{-k})\right) P_{\cH_p}, 
\end{equation}
with some $C_{r,k}(f_1,\ldots,f_k)\in C^\infty(X)$, where each $C_{r,k}$, $r=0,\ldots,k-1,$ is a $k$-multilinear differential operator of order at most $2r$ and $C_{0,k}(f_1,\ldots,f_k)=f_1\ldots f_k$. Using \eqref{e:nondeg0}, one can show (cf. \cite{deleporte1}) that, if at least $l$ of $f_1,\ldots,f_k$ are equal to $h$, then 
\begin{equation}\label{e:Crk}
|C_{r,k}(f_1,\ldots,f_k)|\leq Ch^{l-r}.
\end{equation}
On the other hand, using the asymptotic expansion \eqref{e:Tp} for $T_p$, one can see that each coefficient $f_{j,k}$ in \eqref{e:Tpk} is a linear combination of terms of the form $C_{r,k}(f_1,\ldots,f_k)$, $r\leq j$, where each $f_m$, $m=1,\ldots,k$, equals $g_r$ with some $r=r_m\in \NN$ and at least $k-j+r$ of $f_1,\ldots,f_k$ are equal to $g_0=h$. By \eqref{e:Crk}, this implies \eqref{e:fjk}.

Next, we prove by induction on $k$ that. for any $c>0$ and $k\in \NN$, there exists $C>0$ such that for any $p\in \NN$ and $u_p\in \mathcal H_p$, satisfying \eqref{e:Tp1} and \eqref{e:lp1},
\begin{equation}\label{e:hk}
|(u_p,h^ku_p)|\leq Cp^{-k}. 
\end{equation}
For $k=1$, we have 
\begin{multline*}
(u_p,hu_p)=(u_p,T_pu_p)-(u_p,(T_p-P_{\cH_p}hP_{\cH_p})u_p)\\ =\lambda_p+(u_p,\cO(p^{-1})u_p)=\cO(p^{-1}). 
\end{multline*}
The induction step goes exactly as in the proof of \cite[Proposition 3.1]{deleporte1}, using \eqref{e:Tpk} and \eqref{e:fjk}, that proves \eqref{e:hk}.

Finally, if $x\not\in V_\delta$, then $d(x, U_0)>\delta$ and, by \eqref{e:nondeg0},
$h(x)\geq \alpha \delta^{2}$. Therefore, we have
\begin{multline}\label{e:hk2}
(u_p,h^ku_p)=\int_X h(x)^k |u_p(x)|^2dv_X(x) \geq \int_{X\setminus V_\delta} h(x)^k |u_p(x)|^2dv_X(x)\\ \geq \alpha^k \delta^{2k}\int_{X\setminus V_\delta} |u_p(x)|^2dv_X(x).
\end{multline}
From \eqref{e:hk} and \eqref{e:hk2}, we get \eqref{e:minusVdelta}.
\end{proof}

In the end of this section, we establish some exponential decay estimates for eigensections, corresponding to low-lying eigenvalues, using the methods of Section~\ref{s:tunnel}. As mentioned above, this result turns out to be insufficient for our applications. 

So let $\{T_p\}$ be a self-adjoint Toeplitz operator \eqref{e:Tp}, satisfying the condition \eqref{e:minh=0}. Here we assume more generally that, for some $k\geq 1$, there exists $C>0$ such that
\begin{equation}\label{e:nondeg}
h(x)\geq Cd(x,U_0)^{2k}, \quad x\in X.
\end{equation}

\begin{thm}\label{low-l}
Suppose that $\lambda_p$ is a sequence of eigenvalues of $T_{p}$, satisfying the estimate 
\[
\lambda_p<\frac{C_0}{p^{2k/(2k+1)}}, \quad p\in \NN,
\] 
with some $C_0>0$, independent of $p$, and $u_p\in \mathcal H_p$ is the corresponding normalized eigensection:
\[
T_{p}u_p=\lambda_pu_p, \quad \|u_p\|=1.
\]
For any $c<\alpha_0$, there exists $C_1>0$ such that, for any $p\in \NN$,
\[
\int_X e^{2cp^{1/(2k+1)}d(x,U_0)}|u_p(x)|^2 dv_X(x)<C_1.
\]
\end{thm}

\begin{proof} Let $\rho(x)=d(x,U_0)$. Take a sequence $\tilde \rho_p\in C^\infty(X), p\in \NN,$ satisfying \eqref{(1.1)} and \eqref{dist}. We apply 
Proposition~\ref{p:weighted-est} with $\alpha=cp^{-(2k-1)/(4k+2)}$ for the given $c>0$. We conclude that 
there exists $M>0$ such that, for any $p\in \NN$, 
\begin{equation}\label{e:ineq01}
\int_X e^{2cp^{1/(2k+1)}\tilde \rho_p(x)}\left(h(x)-\frac{M}{p^{2k/(2k+1)}}\right)|u_p(x)|^2 dv_X(x)\leq 0.
\end{equation}

Now the proof is completed similar to the proof of Theorem~\ref{t:forbidden} and we omit it. 
\end{proof}

\section{Asymptotic expansions of low-lying eigenvalues}\label{s:3}

This section is devoted to the proofs of Theorems~\ref{t:DelB} and \ref{t:DelA}.

\subsection{Characterization of Toeplitz operators}\label{s:char}
In this section, we recall the description of Toeplitz operators in terms of their Schwartz kernels introduced in \cite{ma-ma:book,ma-ma08}. 

First, we introduce normal coordinates near an arbitrary point $x_0\in X$, which we fix. Let $a^X$ be the injectivity radius of $(X,g)$. We will identify $B^{T_{x_0}X}(0,a^X)$ with $B^{X}(x_0,a^X)$ by the exponential map $\operatorname{exp}^X : T_{x_0}X\to X$. For $Z\in B^{T_{x_0}X}(0,a^X)$ we identify $L_Z$ to $L_{x_0}$ by parallel transport with respect to the connection $\nabla^L$ along the curve $\gamma_Z : [0,1]\ni u \to \exp^X_{x_0}(uZ)$. Consider the line bundle $L_0$ with fibers $L_{x_0}$ on $T_{x_0}X$. Denote by $\nabla^L$, $h^L$ the connection and the metric on the restriction of $L_0$ to $B^{T_{x_0}X}(0,a^X)$ induced by the identification $B^{T_{x_0}X}(0,a^X)\cong B^{X}(x_0,a^X)$ and the trivialization of $L$ over $B^{T_{x_0}X}(0,a^X)$. 

Let $dv_{TX}$ be the Riemannian volume form of $(T_{x_0}X, g^{T_{x_0}X})$ and $dv_X$ the volume form on $B^{T_{x_0}X}(0,a^X)$, corresponding to the Riemannian volume form $dv_X$ on $B^{X}(x_0,a^X)$ under identification $B^{T_{x_0}X}(0,a^X)\cong B^{X}(x_0,a^X)$. Let $\kappa_{x_0}$ be the smooth positive function on $B^{T_{x_0}X}(0,a^X)$ defined by the equation
\[
dv_{X}(Z)=\kappa_{x_0}(Z)dv_{TX}(Z), \quad Z\in B^{T_{x_0}X}(0,a^X). 
\] 

The almost complex structure $J_{x_0}$ defined in \eqref{e:defJ} induces a splitting $T_{x_0}X\otimes_{\mathbb R}\mathbb C=T^{(1,0)}_{x_0}X\oplus T^{(0,1)}_{x_0}X$, where $T^{(1,0)}_{x_0}X$ and $T^{(0,1)}_{x_0}X$ are the eigenspaces of $J_{x_0}$ corresponding to eigenvalues $i$ and $-i$ respectively. Denote by $\det_{\mathbb C}$ the determinant function of the complex space $T^{(1,0)}_{x_0}X$. Put
\[
\mathcal J_{x_0}=-2\pi i J_0.
\]
Then $\mathcal J_{x_0} : T^{(1,0)}_{x_0}X\to T^{(1,0)}_{x_0}X$ is positive, and $\mathcal J_{x_0} : T_{x_0}X\to T_{x_0}X$ is skew-adjoint. 
Put (see \cite{ma-ma:book,ma-ma08})
\begin{multline} \label{e:defP0}
\mathcal P(Z, Z^\prime)\\ =\frac{\det_{\mathbb C}\mathcal J_{x_0}}{(2\pi)^n}\exp\left(-\frac 14\langle (\mathcal J^2_{x_0})^{1/2}(Z-Z^\prime), (Z-Z^\prime)\rangle +\frac 12 \langle \mathcal J_{x_0} Z, Z^\prime \rangle \right).
\end{multline}
It is the Bergman kernel of the operator $\mathcal L_{x_0}$ on $C^\infty(T_{x_0}X)$ given by \eqref{e:defL0}, that is, the smooth kernel with respect to $dv_{TX}(Z)$ of the orthogonal projection $\mathcal P=\mathcal P_{x_0}$ in $L^2(T_{x_0}X)$ to the kernel of $\mathcal L_{x_0}$.

We choose an orthonormal basis $\{w_j : j=1,\ldots,n\}$ of $T^{(1,0)}_{x_0}X$, consisting of eigenvectors of $\mathcal J_{x_0}$:
\[
\mathcal J_{x_0}w_j=a_jw_j,  \quad j=1,\ldots,n,
\]
with some $a_j>0$.
Then $e_{2j-1}=\frac{1}{\sqrt{2}}(w_j+\bar w_j)$ and $e_{2j}=\frac{i}{\sqrt{2}}(w_j-\bar w_j)$, $j=1,\ldots,n$, form an orthonormal basis of $T_{x_0}X$. We use this basis to define the coordinates $Z$ on $T_{x_0}X\cong \mathbb R^{2n}$ as well as the complex coordinates $z$ on $\mathbb C^{n} \cong \mathbb R^{2n}$, $z_j=Z_{2j-1}+iZ_{2j}, j=1,\ldots,n$.  In this coordinates, we get
\begin{equation}\label{e:defP} 
\mathcal P(Z, Z^\prime)=\frac{1}{(2\pi)^n}\prod_{j=1}^na_j \exp\left(-\frac 14\sum_{k=1}^na_k(|z_k|^2+|z_k^\prime|^2- 2z_k\bar z_k^\prime) \right).
\end{equation}

Let $\{\Xi_p\}_{p\in \mathbb N}$ be a sequence  of linear operators $\Xi_p : L^2(X,L^p)\to L^2(X,L^p)$ with smooth kernel $\Xi_p(x,x^\prime)$ with respect to $dv_X$. Under our trivialization, $\Xi_p(x,x^\prime)$ induces a smooth function $\Xi_{p,x_0}(Z,Z^\prime)$ on the set of all $Z,Z^\prime\in T_{x_0}X$ with $x_0\in X$ and $|Z|, |Z^\prime|<a_X$.  

\begin{defn}[\cite{ma-ma:book,ma-ma08}]
We say that 
\[
p^{-n}\Xi_{p,x_0}(Z,Z^\prime)\cong \sum_{r=0}^k(Q_{r,x_0}\mathcal P_{x_0})(\sqrt{p}Z,\sqrt{p}Z^\prime)p^{-\frac{r}{2}}+\mathcal O(p^{-\frac{k+1}{2}})
\]
with some $Q_{r,x_0}\in \mathbb C[Z,Z^\prime]$, $0\leq r\leq k$,  depending smoothly on the parameter $x_0\in X$, if there exist $\varepsilon\in (0,a_X]$ and $C_0>0$ with the following property:
for any $l\in \mathbb N$, there exist $C>0$ and $M>0$ such that for any $x_0\in X$, $p\geq 1$ and $Z,Z^\prime\in T_{x_0}X$, $|Z|, |Z^\prime|<\varepsilon$, we have 
\begin{multline*}
\Bigg|p^{-n}\Xi_{p,x_0}(Z,Z^\prime)\kappa_{x_0}^{\frac 12}(Z)\kappa_{x_0}^{\frac 12}(Z^\prime) -\sum_{r=0}^k(Q_{r,x_0}\mathcal P_{x_0})(\sqrt{p} Z, \sqrt{p}Z^\prime)p^{-\frac{r}{2}}\Bigg|_{\mathcal C^{l}(X)}\\ 
\leq Cp^{-\frac{k+1}{2}}(1+\sqrt{p}|Z|+\sqrt{p}|Z^\prime|)^M\exp(-\sqrt{C_0p}|Z-Z^\prime|)+\mathcal O(p^{-\infty}).
\end{multline*}
\end{defn}

Observe that if $P_{p,x_0}(Z,Z^\prime)$ denotes the Schwartz kernel of the Bergman projection $P_{\mathcal H_p}$ in the normal coordinates near $x_0\in X$, then, by \cite[Theorem 1.1]{bergman}, for any $k\in \mathbb N$, 
\[
p^{-n}P_{p,x_0}(Z,Z^\prime)\cong
\sum_{r=0}^k(J_{r,x_0}\mathcal P_{x_0})(\sqrt{p} Z, \sqrt{p}Z^\prime)p^{-\frac{r}{2}}+\mathcal O(p^{-\frac{k+1}{2}}),
\]
$J_{r,x_0}(Z,Z^\prime)$ are polynomials in $Z, Z^\prime$, depending smoothly on $x_0$, with the same parity as $r$ and $\operatorname{deg} J_{r,x_0}\leq 3r$. This estimate was introduced in \cite{dai-liu-ma} for the spin$^c$ Dirac operator, see also \cite{ma-ma:book,ma-ma08} for the K\"ahler case and \cite{LMM} for the renormalized Bochner Laplacian. 

Let $f\in C^\infty(X)$. The Schwartz kernel of $T_{f,p}$ is given by 
\[
T_{f,p}(x,x^\prime)=\int_X P_p(x,x^{\prime\prime})f(x^{\prime\prime})P_p(x^{\prime\prime},x^{\prime})dv_X(x^{\prime\prime}). 
\] 
Therefore, for any $x_0\in X$ and $k\in \mathbb N$, we have \cite[Lemma 6.4]{bergman} (see also \cite{ma-ma:book,ma-ma08} for the spin$^c$ Dirac operator and the K\"ahler case and \cite{ioos-lu-ma-ma} for the renormalized Bochner Laplacian)
\[
p^{-n}T_{f,p,x_0}(Z,Z^\prime)\cong \sum_{r=0}^k(Q_{r,x_0}(f)\mathcal P_{x_0})(\sqrt{p}Z,\sqrt{p}Z^\prime)p^{-\frac{r}{2}}+\mathcal O(p^{-\frac{k+1}{2}}),
\]
where the polynomials $Q_{r,x_0}(f)(Z,Z^\prime)$ have the same parity as $r$. 

The coefficients $Q_{r,x_0}(f)$ are computed explicitly in \cite{ma-ma:book,ma-ma08} as follows. For any polynomial $F\in \mathbb C[Z,Z^\prime]$, denote by $F\mathcal P_{x_0}$ the operator in $L^2(T_{x_0}X)$ with smooth kernel $(F\mathcal P_{x_0})(Z,Z^\prime)$.  For polynomials $F,G\in \mathbb C[Z,Z^\prime]$, define the polynomial $\mathcal K[F,G]\in \mathbb C[Z,Z^\prime]$ by the condition
\[
((F\mathcal P_{x_0})\circ (G\mathcal P_{x_0}))(Z,Z^\prime)=(\mathcal K[F,G]\mathcal P_{x_0})(Z,Z^\prime), 
\]
where $((F\mathcal P_{x_0})\circ (G\mathcal P_{x_0}))(Z,Z^\prime)$ is the smooth kernel of the composition $(F\mathcal P_{x_0})\circ (G\mathcal P_{x_0})$ of the operators $F\mathcal P_{x_0}$ and $G\mathcal P_{x_0}$ in $L^2(T_{x_0}X)$.

Then we have
\[
Q_{r,x_0}(f)=\sum_{r_1+r_2+|\alpha|=r}\mathcal K\left[J_{r_1,x_0}, \frac{\partial^\alpha f_{x_0}}{\partial Z^\alpha}(0)\frac{Z^\alpha}{\alpha!}J_{r_2,x_0}\right],
\]
where $f_{x_0}$ is the smooth function on $B^{T_{x_0}X}(0,a^X)$, corresponding to $f$ by the identification $B^{T_{x_0}X}(0,a^X)\cong B^X(x_0,a^X)$.  

For the first three coefficients, we get the following expressions:
\begin{align}
Q_{0,x_0}(f)=& f(x_0),\label{e:Q0f}\\
Q_{1,x_0}(f)=& f(x_0)J_{1,x_0}+\mathcal K\left[J_{0,x_0}, \sum_{j=1}^{2n} \frac{\partial f_{x_0}}{\partial Z_j}(0)Z_jJ_{0,x_0}\right], \label{e:Q1f} 
\end{align}
and 
\begin{multline}
Q_{2,x_0}(f)\\ =f(x_0)(J_{2,x_0}+\mathcal K[J_{1,x_0}, J_{1,x_0}]) +\mathcal K\left[J_{0,x_0}, \sum_{j=1}^{2n}\frac{\partial f_{x_0}}{\partial Z_j}(0)Z_j J_{1,x_0}\right]\\ +\mathcal K\left[J_{1,x_0}, \sum_{j=1}^{2n}\frac{\partial f_{x_0}}{\partial Z_j}(0)Z_j J_{0,x_0}\right]+\mathcal K\left[J_{0,x_0}, \sum_{|\alpha|=2} \frac{\partial^\alpha f_{x_0}}{\partial Z^\alpha}(0)\frac{Z^\alpha}{\alpha!}J_{0,x_0}\right].\label{e:Q2f}
\end{multline}

\begin{lem}\label{l:Tp}
If $f(x_0)=0$ and $df(x_0)=0$, then 
\begin{multline*}
p^{-n}T_{f,p,x_0}(Z,Z^\prime)\\ \cong p^{-1} \left[\mathcal P_{x_0} \left(\frac 12 {\rm Hess}\,f(x_0)Z, Z\right) \mathcal P_{x_0}\right](\sqrt{p}Z,\sqrt{p}Z^\prime)+\mathcal O(p^{-\frac{3}{2}}).
\end{multline*}
\end{lem}

\begin{proof}
First of all, by \eqref{e:Q0f} and \eqref{e:Q1f}, we have
\[
Q_{0,x_0}(f)=Q_{1,x_0}(f)=0.
\]
Observe that
\[
 \sum_{|\alpha|=2} \frac{\partial^\alpha f_{x_0}}{\partial Z^\alpha}(0)\frac{Z^\alpha}{\alpha!}=\left(\frac 12 {\rm Hess}\,f(x_0)Z, Z\right). 
\]
Since $J_{0,x_0}(Z,Z^\prime)=1$, by \eqref{e:Q2f}, we get
\[
Q_{2,x_0}(f)=\mathcal K\left[1, \left(\frac 12 {\rm Hess}\,f(x_0)Z, Z\right)\right]. 
\]
By definition of $\mathcal K$, the function
\[
(Q_{2,x_0}(f) \mathcal P_{x_0})(Z,Z^\prime)=\mathcal K\left[1, \left(\frac 12 {\rm Hess}\,f(x_0)Z, Z\right)\right]\mathcal P_{x_0}(Z,Z^\prime)
\]
is the smooth kernel of the operator $\mathcal P_{x_0} \left(\frac 12 {\rm Hess}\,f(x_0)Z, Z\right) \mathcal P_{x_0}$.
\end{proof}

We have the following criterion for Toeplitz operators \cite[Theorem 6.5]{bergman} (see also \cite{ioos-lu-ma-ma}). This type of criterion was introduced in \cite[Theorem 4.9]{ma-ma08}. 

\begin{thm}\label{t:char}
A family $\{T_p: L^2(X,L^p)\to L^2(X,L^p)\}$ of bounded linear operators is a Toeplitz operator if and only if it satisfies the following three conditions:
\begin{description}
\item[(i)] For any $p\in \mathbb N$, 
\[
T_p=P_{\mathcal H_p}T_pP_{\mathcal H_p}. 
\]
\item[(ii)] For any $\varepsilon_0>0$, there exist $C>0$ and $c>0$ such that for any $p\geq 1$ and $(x,x^\prime)\in X\times X$ with $d(x,x^\prime)>\varepsilon_0$,
\[
|T_{f,p}(x,x^\prime)|\leq C e^{-c\sqrt{p} d(x, x^\prime)}. 
\]
\item[(iii)] There exists a family of polynomials $\mathcal Q_{r,x_0}\in \CC[Z,Z^\prime]$, $r\in \NN$, depending smoothly on $x_0$, of the same parity as $r$ such that, for any $k\in \mathbb N$ and $x_0\in X$,  
\[
p^{-n}T_{p,x_0}(Z,Z^\prime)\cong \sum_{r=0}^k(\mathcal Q_{r,x_0}\mathcal P_{x_0})(\sqrt{p}Z,\sqrt{p}Z^\prime)p^{-\frac{r}{2}}+\mathcal O(p^{-\frac{k+1}{2}}).
\] 
\end{description}
\end{thm} 

Moreover, it is shown in the proof of this theorem (see Proposition 4.11 and (4..30) in \cite{ma-ma08}) that, in this case, for all $x_0\in X$ and $Z,Z^\prime\in T_{x_0}X$,
\[
\mathcal Q_{0,x_0}(Z,Z^\prime)=\mathcal Q_{0,x_0}(0,0),
\]
and the principal symbol $g_0$ of $T_p$ is given by 
\begin{equation}\label{e:prinTp}
g_0(x_0)=\mathcal Q_{0,x_0}(0,0).
\end{equation}

\subsection{Construction of approximate eigensections}
We start the proofs of Theorems~\ref{t:DelB} and \ref{t:DelA} with a construction of approximate eigensections of the operator $T_p$. It follows closely the construction given in \cite[Proposition 4.2 and Proposition 5.1]{deleporte1} in the case of K\"ahler manifolds and  heavily relies on the asymptotic expansions of kernels of Toeplitz operators. So we just give a sketch of this construction and state the main results. We will use notation introduced in Section~\ref{s:asymp-intro}. 

Let $T_p$ be a self-adjoint Toeplitz operator \eqref{e:Tp}, satisfying \eqref{e:minh=0} and \eqref{e:expT-h}. Suppose that $x_0\in U_0$ is a non-degenerate minimum of $h$. The approximate eigensections, which we are going to construct, will be supported in a small neighborhood of $x_0$. So we will use the normal coordinates near $x_0$ and the trivialization of the line bundle $L$ constructed in Section \ref{s:char}. By Theorem~\ref{t:char} and Lemma \ref{l:Tp}, the smooth kernel $T_{p,x_0}(Z,Z^\prime)$ of $T_p$ in these coordinates admits the following asymptotic expansion: for any $k\geq 2$, 
\begin{equation}\label{e:T-Tk}
p^{-n}T_{p,x_0}(Z,Z^\prime)\cong \sum_{r=2}^k(\mathcal Q_{r,x_0}\mathcal P_{x_0})(\sqrt{p}Z,\sqrt{p}Z^\prime)p^{-\frac{r}{2}}+\mathcal O(p^{-\frac{k+1}{2}}),
\end{equation}
where, for any $r$, $\mathcal Q_{r,x_0}\in \CC[Z,Z^\prime]$ is a polynomial of the same parity as $r$ and 
\[
(\mathcal Q_{2,x_0}\mathcal P_{x_0})(Z,Z^\prime)=\left[\mathcal P_{x_0} (q_{x_0}+g_1(x_0)) \mathcal P_{x_0}\right](Z,Z^\prime)=\mathcal T_{x_0}(Z,Z^\prime).
\]
In particular, when $k=2$, we have
\begin{equation}\label{e:T-T2}
p^{-n}T_{p,x_0}(Z,Z^\prime)\cong \mathcal T_{x_0} (\sqrt{p}Z,\sqrt{p}Z^\prime)p^{-1}+\mathcal O(p^{-\frac{3}{2}}),
\end{equation}

First, we construct a formal eigensection of the operator $T_p$. We write formal asymptotic expansions in powers of $p^{-1/2}$:
\[
u_p(Z)=\sum_{j=0}^{+\infty} p^nu^{(j)}(\sqrt{p}Z)p^{-j/2},\quad 
\lambda_p=p^{-1}\sum_{j=0}^{+\infty} \lambda^{(j)}p^{-j/2}
\]
and express the cancellation of the coefficients of $p^{-j/2}$ in the formal expansion for $(T_p-\lambda_p)u_p$ step by step, using the asymptotic expansions \eqref{e:T-Tk} and \eqref{e:T-T2}.

At the first step, we get:
\[
\mathcal T_{x_0}u^{(0)}=\lambda^{(0)}u^{(0)}.
\]
Thus, $\lambda^{(0)}=\lambda$ is an eigenvalue of $\mathcal T_{x_0}$ and $u^{(0)}$ is an associated eigenfunction. 

At the second step, we get:
\[
\mathcal T_{x_0}u^{(1)}+\mathcal Q_{1,x_0}\mathcal P_{x_0}u^{(0)}=\lambda^{(0)}u^{(1)}+\lambda^{(1)}u^{(0)}.
\]
One can show that this equation has a solution $(u^{(1)}, \lambda^{(1)})$. 

At the third step, we get an equation, which can't be solved in general. One can solve it in the case when $\lambda$ is a simple eigenvalue. In this case, we can proceed further and obtain a solution as a formal asymptotic series in powers of $p^{-1/2}$. To obtain approximate eigensections, we multiple the formal eigenfunctions by appropriate cut-off functions.

We arrived at the following statements. 

\begin{prop}\label{p:quasi1}
Let $x_0\in U_0$ be a non-degenerate minimum of $h$ and $\lambda$ be an eigenvalue of $\cT_{x_0}$ of multiplicity $m$. There exist an orthonormal system $u^{(1)}, \ldots, u^{(m)}$ from $\mathcal S(T_{x_0}X)$ of eigenfunctions of $\cT_{x_0}$ with eigenvalue $\lambda$, functions $v^{(1)}, \ldots, v^{(m)}$ from $\mathcal S(T_{x_0}X)$ and real numbers $\mu^{(1)}, \ldots, \mu^{(m)}$ such that if for each $p$ and $j=1,\ldots,m$ we define $u^{(j)}_{p}\in L^2(X,L^p)$ and $\lambda^{(j)}_{p}\in \RR$ by 
\[
u^{(j)}_{p}(Z)=p^n \chi(Z) (u^{(j)}(\sqrt{p}Z)+p^{-1/2}v^{(j)}(\sqrt{p}Z)),\quad 
\lambda^{(j)}_{p}=p^{-1}\lambda+p^{-1/2}\mu^{(j)},
\]
$\chi$ is a cut-off function from $C^\infty_c(B^{T_{x_0}X}(0,\varepsilon))\cong C^\infty_c(B^X(x_0,\varepsilon))$ with some $\varepsilon\in (0,a_X/4)$, then
\[
\|T_pu^{(j)}_{p}-\lambda^{(j)}_{p}u^{(j)}_{p}\|_{L^2(X,L^p)}=\mathcal O(p^{-2}).
\]
\end{prop}

\begin{prop}\label{p:quasi2}
Let $x_0\in U_0$ be a non-degenerate minimum of $h$, $\lambda$ be a simple eigenvalue of $\cT_{x_0}$ and $u_0$ be a normalized eigenfunction of $\cT_{x_0}$. There exist a sequence $(u_k)_{k\geq 1}$ of functions from $\mathcal S(T_{x_0}X)$ with $(u_0,u_k)=\delta_{0k}$ and a sequence $(\lambda_k)_{k\geq 0}$ of real numbers, with $\lambda_0=\lambda$ and $\lambda_k=0$ for $k$ odd, such that if for each $N$ and $p$ we define $u_{N,p}\in L^2(X,L^p)$ and $\lambda_{N,p}\in \RR$ by 
\[
u_{N,p}(Z)=\chi(Z)  p^n\sum_{k=0}^N p^{-k/2}u_k(\sqrt{p}Z),\quad 
\lambda_{N,p}=p^{-1}\sum_{k=0}^N p^{-k/2}\lambda_k,
\]
$\chi$ is a cut-off function from $C^\infty_c(B^{T_{x_0}X}(0,\varepsilon))\cong C^\infty_c(B^X(x_0,\varepsilon))$ with some $\varepsilon\in (0,a_X/4)$,  then
\[
\|T_pu_{N,p}-\lambda_{N,p}u_{N,p}\|_{L^2(X,L^p)}=\mathcal O(p^{-(N+3)/2})
\]
\end{prop}

\subsection{Proofs of Theorems~\ref{t:DelB} and \ref{t:DelA}} 
Suppose that $T_p$ is a self-adjoint Toeplitz operator with the principal symbol $h$, satisfying \eqref{e:minh=0}, such that each minimum is non-degenerate. In this case, $U_0=\{x_1,\ldots, x_N\}$ and the model operator $\mathcal T$ associated with $T_p$ is a self-adjoint operator on $L^2(T_{x_1}X)\oplus \ldots \oplus  L^2(T_{x_N}X)$ defined by \eqref{e:def-Tau}.

\begin{lem}\label{l:sp}
For any $c>0$ and $\gamma\in (0,1/2)$, there exists $C>0$ such that for any sequence $u_p\in \cH_p$ of eigensections of $T_p$:
\[
T_pu_p=\lambda_pu_p, \quad \|u_p\|=1,
\] 
with $\lambda_p\in (0, cp^{-1})$, we have for $p\in \NN$,  
\[
(p\lambda_p-Cp^{-1/2+\gamma}, p\lambda_p+Cp^{-1/2+\gamma})\cap \sigma(\cT)\neq \emptyset. 
\]
\end{lem}

\begin{proof}
Fix $c>0$. By Theorem~\ref{low-l0},  for any $\delta>0$ and $k\in \NN$, there exists $C_1>0$ such that, for any sequence $u_p\in \cH_p$ of normalized eigensections of $T_p$ with $\lambda_p\in (0, cp^{-1})$ and for any $p\in \NN$,  
\begin{equation}\label{e:loc-xj}
\int_{X\setminus V_\delta} |u_p(x)|^2 dv_X(x)\leq C_1 p^{-k}.
\end{equation}
Using a partition of unity subordinate to the open covering 
\[
X=\bigcup_{j=1}^NB^X(x_j, \varepsilon_0) \cup (X\setminus U_0)
\] 
with some $\varepsilon_0 \in (0,a_X)$ small enough, we can write 
\[
u_p=\sum_{j=1}^N u_p^{(j)}+u_p^{(0)},
\]
where each $u_p^{(j)}$ is supported in $B^X(x_j, \varepsilon_0)\cong B^{T_{x_j}X}(0,\varepsilon_0)$ and satisfies \eqref{e:loc-xj}, $u_p^{(0)}$ is supported in $X\setminus U_0$ and, by \eqref{e:loc-xj}, $\|u_p^{(0)}\|=\mathcal O(p^{-\infty})$. We will assume that the balls $B^X(x_j, \varepsilon_0)$ are pairwise disjoint.  

We will use the same notation $u_p^{(j)}$ for the corresponding function on $B^{T_{x_j}X}(0,\varepsilon_0)$ and define a function $\tilde u_p^{(j)}\in C^\infty_c(B^{T_{x_j}X}(0,\varepsilon_0))$ by 
\[
\tilde u_p^{(j)}(Z)=\kappa_{x_j}^{1/2}(Z) u_p^{(j)}(Z), \quad Z\in B^{T_{x_j}X}(0,\varepsilon_0). 
\]
Put
\begin{equation}\label{e:tildeu_p}
\tilde u_p=\bigoplus_{j=1}^N\tilde u_p^{(j)}\in \bigoplus_{j=1}^N L^2(T_{x_j}X).
\end{equation}
Observe that 
\begin{align*}
\|\tilde u_p^{(j)}\|_{L^2(T_{x_j}X)}& =\|\tilde u_p^{(j)}\|_{L^2(B^{T_{x_j}X}(0,\varepsilon_0))}=\|u_p^{(j)}\|_{L^2(B^{X}(x_j,\varepsilon_0))}=\|u_p^{(j)}\|, \\
\|\tilde u_p\|& =1+\mathcal O(p^{-\infty}).
\end{align*}

For any $j=1,\ldots,N$ and $p\in \NN$, let $\mathcal T_{x_j,p}$ be the operator in $L^2(T_{x_j}X)$ with the Schwartz kernel 
\[
\mathcal T_{x_j,p}(Z,Z^\prime)=p^n\mathcal T_{x_j}(\sqrt{p}Z,\sqrt{p}Z^\prime), \quad Z,Z^\prime\in T_{x_j}X,
\]
where $\mathcal T_{x_j}(Z,Z^\prime)$ is the Schwartz kernel of the operator $\mathcal T_{x_j}$. It is clear that the operator $\mathcal T_{x_j,p}$ is obtained from the operator $\mathcal T_{x_j}$ by an obvious scaling. So these operators are unitarily equivalent and, in particular, have the same spectrum. Let $\mathcal T_p$ be the self-adjoint operator on $L^2(T_{x_1}X)\oplus \ldots \oplus  L^2(T_{x_N}X)$ defined by 
\[
\mathcal T_p=\mathcal T_{x_1,p}\oplus \ldots \oplus \mathcal T_{x_N,p}.
\] 

Fix some $\gamma \in (0,1/2)$. Using \eqref{e:T-T2} and \eqref{e:loc-xj}, one can show (cf. \cite[Proposition 2.7]{deleporte1}) that there exists $C_2>0$ such that for any $p\in \NN$ and $j=1, \ldots,N$, 
\begin{multline}\label{e:T-cT}
\|\kappa_{x_j}^{1/2}T_p u_p^{(j)}-p^{-1}\cT_{x_j,p}\tilde u_p^{(j)}\|_{L^2(B^{T_{x_j}X}(0,\varepsilon_0+\varepsilon))}\\ \leq C_2p^{-3/2+\gamma}\|\tilde u_p^{(j)}\|_{L^2(T_{x_j}X)}. 
\end{multline}
Here $\varepsilon>0$ is given by the asymptotic expansion \eqref{e:T-T2}. We choose $\varepsilon_0>0$ in such a way that $\varepsilon_0+\varepsilon\in (0,a_X)$. The constant $C_2>0$ depends only on the constant $C_1>0$ in \eqref{e:loc-xj} and, therefore, on $c>0$, but not on the particular sequence $u_p$. 

Since each $\tilde u_p^{(j)}$ is supported in $B^{T_{x_j}X}(0,\varepsilon_0)$ and the kernel of $\cT_{x_j,p}$ is rapidly decaying outside the diagonal as $p\to \infty$, we have
\[
\|\cT_{x_j,p} \tilde u_p^{(j)}\|_{L^2(T_{x_j}X\setminus B^{T_{x_j}X}(0,\varepsilon_0+\varepsilon))}=\mathcal O(p^{-\infty})\|\tilde u_p^{(j)}\|,
\]
and, therefore, 
\begin{multline*}
\|p^{-1}\cT_p\tilde u_p-\lambda_p\tilde u_p\|^2=\sum_{j=1}^N \|p^{-1}\cT_{x_j,p} \tilde u_p^{(j)}-\lambda_p \tilde u_p^{(j)}\|^2_{L^2(T_{x_j}X)}\\ =\sum_{j=1}^N \|p^{-1}\cT_{x_j,p}\tilde u_p^{(j)}-\lambda_p\tilde u_p^{(j)}\|^2_{L^2(B^{T_{x_j}X}(0,\varepsilon_0+\varepsilon))}+\mathcal O(p^{-\infty})\|\tilde u_p^{(j)}\|^2.
\end{multline*}
By \eqref{e:T-cT}, it follows that
\begin{multline}\label{e:Tp0}
\|p^{-1}\cT_p\tilde u_p-\lambda_p\tilde u_p\|^2\\ 
\begin{aligned}
& =\sum_{j=1}^N \|\kappa_{x_j}^{1/2} T_{p} u_p^{(j)}-\lambda_p\tilde u_p^{(j)}\|^2_{L^2(B^{T_{x_j}X}(0,\varepsilon_0+\varepsilon))} +C^2_2p^{-3+2\gamma}\|\tilde u_p\|^2 \\ &=\sum_{j=1}^N \|T_{p} u_p^{(j)}-\lambda_pu_p^{(j)}\|^2_{L^2(B^{X}(x_j,\varepsilon_0+\varepsilon))} +C_2^2p^{-3+2\gamma}.
\end{aligned}
\end{multline}

Since each $u_p^{(j)}$ is supported in $B^{X}(x_j,\varepsilon_0)$, we get
\begin{equation}\label{e:Tp1a}
\|T_{p} u_p^{(j)}-\lambda_pu_p^{(j)}\|^2=\|T_{p} u_p^{(j)}-\lambda_pu_p^{(j)}\|^2_{L^2(B^{X}(x_j,\varepsilon_0+\varepsilon))}+\mathcal O(p^{-\infty}).
\end{equation}
Moreover, since $B^X(x_j, \varepsilon_0)$ are pairwise disjoint, $(u_p^{(j)}, u_p^{(k)})_{L^2(X)}=0$ for $j\neq k$. By Theorem \ref{t:char}, this implies that, for $j\neq k$,
\[
(T_{p} u_p^{(j)}-\lambda_pu_p^{(j)}, T_{p} u_p^{(k)}-\lambda_pu_p^{(k)})=\mathcal O(p^{-\infty}).
\]
Using this almost orthogonality property and the fact that $u_p$ is an eigensection of $T_p$ with the eigenvalue $\lambda_p$, we get
\begin{equation}\label{e:Tp2}
\sum_{j=1}^N \|T_{p} u_p^{(j)}-\lambda_pu_p^{(j)}\|^2 =\|T_{p} u_p-\lambda_pu_p\|^2+\mathcal O(p^{-\infty})=\mathcal O(p^{-\infty}).
\end{equation}
Combining \eqref{e:Tp0}, \eqref{e:Tp1a} and \eqref{e:Tp2}, we conclude that 
\begin{equation}\label{e:quasiCT}
\|p^{-1}\cT_p\tilde u_p-\lambda_p\tilde u_p\|=C_3p^{-3/2+\gamma}\|\tilde u_p\|,
\end{equation}
that completes the proof of the lemma.
\end{proof}

Fix $\gamma\in (0,1/2)$ and $c>0$. By Lemma~\ref{l:sp}, there exists $C>0$ such that, for all $p\in \NN$, 
\[
\sigma(T_p)\cap (0, cp^{-1})\subset \bigcup_{\mu\in \sigma(\cT)\cap (0, c)} (\mu p^{-1}-C p^{-3/2+\gamma}, \mu p^{-1}+C p^{-3/2+\gamma}).
\]

Let us show that, for any $\mu\in \sigma(\cT)$, the number $m_p$ of eigenvalues of $T_p$ in the interval $(\mu p^{-1}-C p^{-3/2+\gamma}, \mu p^{-1}+C p^{-3/2+\gamma})$ is independent of $p$ for large $p$ and equals the multiplicity $m$ of $\mu$:
\begin{equation}\label{e:mpm}
m_p=m, \quad p\gg 1. 
\end{equation}

Using Proposition~\ref{p:quasi1}, one can easily show that, if $\mu$ is an eigenvalue of the operator $\cT$ of multiplicity $m$, then there exists at least $m$ eigenvalues of the operator $T_p$ in the interval $(\mu p^{-1}- Cp^{-3/2}, \mu p^{-1} + Cp^{-3/2})$  for some $C>0$. Therefore, $m_p\geq m$ for any $p$ large enough. 

On the other hand, let $u_{1,p}, \ldots, u_{m_p,p}$ be an orthonormal basis of eigensections of $T_p$ with the corresponding eigenvalues $\lambda_{1,p}, \ldots, \lambda_{m_p,p}$ in the interval $(\lambda p^{-1}-C p^{-3/2+\gamma}, \lambda p^{-1}+C p^{-3/2+\gamma})$. For the functions $\tilde u_{1,p}, \ldots, \tilde u_{m_p,p}$ defined by \eqref{e:tildeu_p}, we get
\[
(\tilde u_{j,p}, \tilde u_{k,p})=\delta_{jk}+\mathcal O(p^{-\infty})
\]
and, by \eqref{e:quasiCT},
\[
\|p^{-1}\cT_p\tilde u_{j,p}-\lambda_{j,p}\tilde u_{j,p}\|=Cp^{-3/2+\gamma}\|\tilde u_p\|.
\]
It follows that there are at least $m_p$ eigenvalues of $\cT_p$ in the interval $(\mu-C p^{-1/2+\gamma}, \mu+C p^{-1/2+\gamma})$. But, for large $p$, the only eigenvalue $\cT_p$ in this interval is $\mu$. Therefore, $m_p\leq m$, that completes the proof of \eqref{e:mpm}. 

This shows that each eigenvalue $\lambda^m_p$ has an asymptotic expansion, as $p\to \infty$, of the form
\[
\lambda^m_p=p^{-1}\mu_m+\mathcal O(p^{-3/2}).
\]
Now the asymptotic expansions \eqref{t:DelB} in Theorem~\ref{t:DelB} and the complete asymptotic expansions in Theorem~\ref{t:DelA} can be proved in a standard way by means of Propositions~\ref{p:quasi1} and~\ref{p:quasi2}.

\section{Applications to the Bochner Laplacian}\label{s:4} 

\subsection{Upper bounds for eigenvalues of the Bochner Laplacian}\label{s:41} 

\begin{proof}[Proof of Theorem~\ref{t:upper-estimates}]
The Bochner Laplacian $\Delta^{L^p}$ can be written in the following Schr\"odinger operator type form
 \[
\Delta^{L^p}=\Delta_p+p\tau. 
 \]
So we have
\[
p^{-1}P_{\mathcal H_p}\Delta^{L^p}P_{\mathcal H_p}=p^{-1} P_{\mathcal H_p}\Delta_pP_{\mathcal H_p}+P_{\mathcal H_p}\tau P_{\mathcal H_p}.
\]
It is well-known (see, for instance, \cite[Theorem XIII.3]{RSIV}), the Rayleigh-Ritz technique) that 
\[
\lambda_j(\Delta^{L^p})\leq \lambda_j(P_{\mathcal H_p}\Delta^{L^p}P_{\mathcal H_p}), \quad j\in \NN,
\]
It remains to get an upper bound for $\lambda_j(P_{\mathcal H_p}\Delta^{L^p}P_{\mathcal H_p})$. 

Here the crucial fact is that the operator $P_{\mathcal H_p}\Delta_pP_{\mathcal H_p}=\Delta_pP_{\mathcal H_p}$ is a Toeplitz operator. This was proved in \cite{bergman}, extending previous results of \cite{ma-ma08a}.  Denote by $P_{1,p}(x,x^\prime)$, $x,x^\prime\in X$ the smooth kernel of the operator $\Delta_pP_{\mathcal H_p}$ with respect to the Riemannian volume form $dv_X$. It is called a generalized Bergman kernel of $\Delta_p$.  For any $k\in \mathbb N$ and $x_0\in X$,   we have  
\begin{multline}\label{e:Pqr}
p^{-n}P_{1,p,x_0}(Z,Z^\prime)\cong 
\sum_{r=2}^jF_{1,r,x_0}(\sqrt{p} Z, \sqrt{p}Z^\prime)\kappa^{-\frac 12}(Z)\kappa^{-\frac 12}(Z^\prime)p^{-\frac{r}{2}+1},
\end{multline}
where
\begin{equation}\label{e:Fqr}
F_{1,r,x_0}(Z,Z^\prime)=J_{1,r,x_0}(Z,Z^\prime)\mathcal P_{x_0}(Z,Z^\prime),
\end{equation}
$J_{1,r,x_0}(Z,Z^\prime)$ are polynomials in $Z, Z^\prime$, depending smoothly on $x_0$, with the same parity as $r$ and $\operatorname{deg} J_{1,r,x_0}\leq 3r$.

So the operator $p^{-1}P_{\mathcal H_p}\Delta^{L^p}P_{\mathcal H_p}$ is a Toeplitz operator:
\[
p^{-1}P_{\mathcal H_p}\Delta^{L^p}P_{\mathcal H_p}=P_{\mathcal H_p}\left(\sum_{l=0}^\infty p^{-l}g_l\right)P_{\mathcal H_p}+\mathcal O(p^{-\infty}),
\]
where $g_0=\tau$ and $g_1$ is the principal symbol of $P_{\mathcal H_p}\Delta_pP_{\mathcal H_p}$. By \eqref{e:prinTp} and \eqref{e:Fqr}, it is  given by 
\begin{equation}\label{e:defJ12}
g_1(x_0)=J_{1,2}(x_0):=J_{1,2,x_0}(0,0)=\frac{F_{1,2,x_0}(0,0)}{\mathcal P_{x_0}(0,0)}.
\end{equation}
By Theorem~\ref{t:DelB}, we have 
\[
p^{-1}\lambda_j(P_{\mathcal H_p}\Delta^{L^p}P_{\mathcal H_p})=\tau_0+p^{-1}\mu_j+p^{-3/2}\phi_j+\mathcal O(p^{-2}), 
\]
with some $\phi_j\in\RR$, that immediately completes the proof of Theorem~\ref{t:upper-estimates}.
\end{proof}

The coefficient $F_{1,2,x_0}$ (and, therefore, the function $J_{1,2}$) can be computed explicitly (see Subsection 2.1, in particular, the formula (2.12) in \cite{ma-ma08a}). Let us recall this formula. We will use notation of Section~\ref{s:char}. Put
\[
\frac{\partial}{\partial z_j}=\frac{1}{2}\left(\frac{\partial}{\partial Z_{2j-1}}-i\frac{\partial}{\partial Z_{2j}}\right), \quad \frac{\partial}{\partial \overline{z}_j}=\frac{1}{2}\left(\frac{\partial}{\partial Z_{2j-1}}+i\frac{\partial}{\partial Z_{2j}}\right).
\]
Let $\mathcal R(Z) = \sum_{j=1}^{2n} Z_j e_j=Z$ denote the radial vector field on $T_{x_0}X$. Define first order differential operators $b_j,b^{+}_j, j=1,\ldots,n,$ on $T_{x_0}X$ by
\[
b_j= -2\nabla_{\tfrac{\partial}{\partial z_j}}-R^L_{x_0}(\mathcal R, \tfrac{\partial}{\partial z_j})=-2{\tfrac{\partial}{\partial z
_j}}+\frac{1}{2}a_j\overline{z}_j,
\]
\[
b^{+}_j= 2\nabla_{\tfrac{\partial}{\partial \overline{z}_j}} + R^L_{x_0}(\mathcal R, \tfrac{\partial}{\partial \overline{z}_j})=2{\tfrac{\partial}{\partial\overline{z}_j}}+\frac{1}{2}a_j z_j.
\]
So we can write
\[
\cL_{x_0}=\sum_{j=1}^n b_j b^{+}_j,\quad \tau(x_0)= \sum_{j=1}^n a_j. 
\]
Then 
\begin{equation}\label{e:F120}
F_{1,2,x_0}(Z,Z^\prime) = [\cP_{x_0}\mathcal F_{1,2,x_0}\cP_{x_0}] (Z,Z^\prime),
\end{equation}
where $\mathcal F_{1,2,x_0}$ is an unbounded linear operator in $L^2(T_{x_0}X)$ given by 
\begin{multline}\label{e:F12}
\mathcal F_{1,2,x_0}
= 4 \left \langle R^{TX}_{x_0} \left(\frac{\partial}{\partial z_j}, \frac{\partial}{\partial z_k}\right) \frac{\partial}{\partial \overline{z}_j}, 
\frac{\partial}{\partial \overline{z}_k}\right \rangle\\
+ \left \langle (\nabla ^{X} \nabla ^{X}\mathcal J)_{(\mathcal R,\mathcal R)} \frac{\partial}{\partial z_j},
 \frac{\partial}{\partial \overline{z}_j} \right \rangle 
+\frac{i}{4} 
\tr_{|TX} \Big(\nabla ^{X} \nabla ^{X}(J\mathcal J)\Big)_{(\mathcal R, \mathcal R)} \\
+ \frac{1}{9} |(\nabla_{\mathcal R}^X \mathcal J) \mathcal R |^2
+ \frac{4}{9} \sum_{j,k=1}^n\left\langle(\nabla ^{X}_{\mathcal R} \mathcal J) \mathcal R, \frac{\partial} {\partial z_j}\right\rangle b^+_j \cL_{x_0}^{-1}b_k 
\left\langle(\nabla ^{X}_{\mathcal R}\mathcal J)\mathcal R,\frac{\partial}
{\partial\overline{z}_k}\right\rangle.
\end{multline}

The formulas \eqref{e:defJ12}, \eqref{e:F120} and \eqref{e:F12} allow one to compute the function $J_{1,2}$. For instance, as shown in \cite[(2.30)]{ma-ma08a}, if $J_0=J$, then
\[
J_{1,2}(x_0)=\frac{1}{24}|\nabla^XJ|^2_{x_0}. 
\]
Here if $\{e_j\}_{j=1,\ldots,2n}$ is a local orthonormal frame of $(TX,g^{TX})$, then 
\[
|\nabla^XJ|^2=\sum_{i,j=1}^{2n} |(\nabla^X_{e_i}J)e_j|^2. 
\]

\subsection{Computation of the spectrum of the model operator}
In this section, we compute the spectrum of the Toeplitz operator $\cT(Q)$ in $L^2(\CC^n)$ given by
\[
\cT(Q)=\cP Q : \ker \mathcal L_{x_0}\subset L^2(\CC^n) \to \ker \mathcal L_{x_0}\subset L^2(\CC^n), 
\]
where $Q=Q(z,\bar z)$ is a polynomial in $\CC^n$ and $\cP$ is the orthogonal projection in $L^2(\CC^n)\cong L^2(T_{x_0}X)$ to the kernel of $\mathcal L_{x_0}$ (see \eqref{e:defP}). We will keep notation of Section~\ref{s:char}. 

Recall that the Fock space $\cF_n$ is the space of holomorphic functions $F$ in $\CC^n$ such that $e^{-\frac{1}{2}|z|^2}F\in L^2(\CC^n)$. It is a closed subspace in $L^2(\CC^n; e^{-\frac{1}{2}|z|^2}dz)$ and the orthogonal projection $\Pi: L^2(\CC^n; e^{-\frac{1}{2}|z|^2}dz)\to \cF_n$ is given by
\[
\Pi F(z)=\frac{1}{\pi^n}\int_{\CC^n}\exp\left(-|z^\prime|^2+z\cdot \bar z^\prime\right) F(z^\prime)\,dz^\prime\,d\bar z^\prime.
\]

Consider the isometry $S : L^2(\CC^n; e^{-\frac{1}{2}|z|^2}dz)\to L^2(\CC^n)$ given, for $u\in L^2(\CC^n; e^{-\frac{1}{2}|z|^2}dz)$, by
\[
Su(z)=\frac{\prod_{j=1}^na_j }{2^n}e^{-\frac{1}{4}\sum_j a_j|z_j|^2} u\left(\phi(z)\right),
\]
where $\phi : \CC^n\to \CC^n$ is a linear isomorphism given by
\[
\phi(z)=\left(\frac{\sqrt{a_1}}{\sqrt{2}}z_1, \ldots, \frac{\sqrt{a_n}}{\sqrt{2}}z_n\right), \quad z\in \CC^n.
\]
It is easy to see that $S\Pi S^{-1}=\cP$. It follows that $S(\cF_n)=\ker \mathcal L_{x_0}$ and
\[
\cT(Q)= S\cT^0(Q\circ \phi^{-1})S^{-1}. 
\]
where $\mathcal T^0(Q)$ is a Toeplitz operator in the Fock space defined by
\[
\mathcal T^0(Q)=\Pi Q : \cF_n\to \cF_n.
\] 
In particular, the spectrum of $\cT(Q)$ in $\ker \mathcal L_{x_0}$ coincides with the spectrum of $\cT^0(Q\circ \phi^{-1})$ in $\cF_n$. 
 
To compute the spectrum of $\mathcal T^0(P)$ in $\cF_n$ for a positive definite quadratic form $P$, we will use the well-known relation between the Bargmann-Fock and the Schr\"odinger representations of the canonical commutation relations via the Bargmann transform. Recall that the Bargmann transform \cite{Bargmann} is an isometry $B: L^2(\RR^n)\to \cF_n$ defined by
\[
Bf(z)=\pi^{-n/4}\int_{\RR^n}\exp\left[-\left(\frac{1}{2}z\cdot z+\frac{1}{2}x\cdot x-\sqrt{2}z\cdot x\right)\right]f(x)dx, \quad z\in \CC^n.
\]

It is easy to check (see, for instance, \cite[(3.15b)]{Bargmann}) that, for the standard position and momentum operators 
\[
\hat q_k=x_k,\quad \hat p_k=\frac{1}{i}\frac{\partial}{\partial x_k},\quad k=1,\ldots,n, 
\]
we have 
\[
B\hat q_kB^{-1}=\frac{1}{\sqrt{2}}\left(z_k+\frac{\partial}{\partial z_k}\right), \quad B\hat p_kB^{-1}=\frac{1}{\sqrt{2}}i\left(z_k-\frac{\partial}{\partial z_k}\right). 
\]
Accordingly, for the creation and annihilation operators
\begin{equation}\label{e:a}
\hat a^*_k=\frac{1}{\sqrt{2}}(\hat q_k-i\hat p_k),\quad \hat a_k=\frac{1}{\sqrt{2}}(\hat q_k+i\hat p_k),
\end{equation}
we get 
\begin{equation}\label{e:BaB}
B\hat a^*_kB^{-1}=z_k, \quad B\hat a_kB^{-1}=\frac{\partial}{\partial z_k}.
\end{equation}

Observe, for any $F\in L^2(\CC^n; e^{-\frac{1}{2}|z|^2}dz)$, 
\begin{equation}\label{e:dPi}
\frac{\partial}{\partial z_k}\Pi F =\Pi \bar z_k F, \quad k=1,\ldots,n,
\end{equation}
and, for any $F\in \cF_n$,
\begin{equation}\label{e:zPi}
z_k\Pi F =\Pi (z_k F)\quad k=1,\ldots,n.
\end{equation}

Let $P$ be a polynomial in $\CC^n$. If we write $P$ as 
\[
P(\bar z,z)= \sum_{k,l\in \ZZ^n_+} A_{k;l} \bar z_1^{k_1}\ldots \bar z_n^{k_n} z_1^{l_1}\ldots z_n^{l_n},
\]
then, using \eqref{e:dPi} and \eqref{e:zPi}, one can easily see that, for any $F\in \cF_n$,
\begin{equation}\label{e:QPi}
P(\partial_z,z) F =\Pi (P(\bar z,z)F)=\cT^0(P)F,
\end{equation}
where $P(\partial_z,z)$ is the operator in $\cF_n$ given by
\[
P(\partial_z,z)=\sum_{k,l\in \ZZ^n_+} A_{k;l}\frac{\partial^{k_1}}{\partial z_1^{k_1}}\ldots \frac{\partial^{k_n}}{\partial z_n^{k_n}} z_1^{l_1}\ldots  z_n^{l_n}.
\]
By \eqref{e:BaB}, it follows that, under the Bargmann transform $B$, the operator $\cT^0(P)$ in $\cF_n$ corresponds to the operator $B^{-1}\cT^0(P)B$ in $L^2(\RR^n)$ given by
\[
B^{-1}\cT^0(P)B=\sum_{k,l\in \ZZ^n_+} A_{k;l} \hat a^{k_1}\ldots \hat a^{k_n} (\hat a_1^*)^{l_1}\ldots  (\hat a_n^*)^{l_n}.
\]
In the terminology of \cite{Berezin71}, the operator $B^{-1}\cT^0(P)B$ is the differential operator with polynomial coefficients in $\RR^n$ with anti-Wick symbol $P(\bar z,z)$. Here, by \eqref{e:a} and \eqref{e:BaB}, the phase space $\RR^{2n}\cong T^*\RR^n$ is identified to $\CC^n$ by the linear isomorphism 
\begin{equation}\label{e:zk-xk}
z_k=\frac{1}{\sqrt{2}}(x_k-i\xi_k), \quad k=1,\ldots,n.
\end{equation}
One can compute the Weyl symbol of this operator by a well-known formula (see, for instance, \cite{Berezin71}). The paper \cite{Berezin71} also contains some sufficient  self-adjointness conditions for the operator $B^{-1}P(\partial_z,z)B$. In particular, if $P$ is a positive definite quadratic form, then (see, for instance, \cite[Proposition 3.6]{deleporte1}), 
\[
B^{-1}\cT^0(P)B=Op_w(\tilde{P})+\frac{\operatorname{tr}(\tilde{P})}{2}, 
\]
where $\tilde P$ is a quadratic form on $\RR^{2n}$, corresponding to $P$ under the linear isomorphism \eqref{e:zk-xk}, $Op_w(\tilde{P})$ is the pseudodifferential operator in $\RR^n$ with Weyl symbol $\tilde{P}$.

Consider the case $n=1$ and $P$ is a positive definite quadratic form on $\CC\cong \RR^2$. Using rotations, we can assume that
\[
P(Z_1,Z_2)=\alpha Z_1^2+\beta Z_2^2. 
\]
The corresponding form $\tilde P$ on $\RR^2$ is given by
\[
\tilde P(x,\xi)=\frac{1}{2}(\alpha x^2+\beta\xi^2). 
\]
We have 
\[
Op_w(\tilde{P})= \frac{1}{2}(\alpha x^2+\beta D_x^2), \quad \operatorname{tr}(\tilde{Q})=\frac{\alpha+\beta}{2},
\]
and the eigenvalues of $B\mathcal T^0(P)B^{-1}$ are
\[
\lambda_j=\frac{\sqrt{\alpha\beta}}{2}(2j+1)+\frac{\alpha+\beta}{4}=\sqrt{\alpha\beta}j+\frac{(\sqrt{\alpha}+\sqrt{\beta})^2}{4}, \quad j=0,1,\ldots.
\]
So, for a general positive definite quadratic form $P$, the eigenvalues of $B\mathcal T^0(P)B^{-1}$ are
\[
\lambda_j=\sqrt{D}j+\frac{A^2}{4}, \quad j=0,1,\ldots,
\]
where $D=\det P$, $A=\operatorname{tr} P^{1/2}$.

Coming back to the initial problem and applying this result when 
\[
P(z)=(Q\circ \phi^{-1})(z)=\frac{2}{a_1}Q(z), \quad z\in \CC\cong \RR^2, 
\] 
where $Q$ is a positive definite quadratic form on $\CC\cong \RR^2$, we find the eigenvalues of $\mathcal T(Q)$:
\begin{equation}\label{e:lambda-T}
\lambda_j=\frac{2\sqrt{D}}{a_1}j+\frac{A^2}{2a_1}, \quad j=0,1,\ldots,
\end{equation}
where $D=\det Q$, $A=\operatorname{tr} Q^{1/2}$.

\subsection{Relation with the magnetic Schr\"odinger operator}
In this section, we recall the relationship of the Bochner Laplacian with the standard magnetic Schr\"odinger operator. 
 
Assume that the Hermitian line bundle $(L,h^L)$ on $X$ is trivial, that is, $L=X\times \CC$ and $|(x,z)|^2_{h^L}=|z|^2$ for $(x,z)\in X\times \CC$. Then the Hermitian connection $\nabla^L$ can be written as $\nabla^L=d-i \mathbf A$ with some real-valued 1-form $\mathbf A$ (a magnetic potential). Its curvature $R^L$ is given by 
\begin{equation}\label{e:RLB}
R^L=-i\mathbf B,
\end{equation}
where $\mathbf B=d\mathbf A$ is a real-valued 2-form (a magnetic field). For the form $\omega$ defined by \eqref{e:def-omega}, we have
\[
\omega=\frac{1}{2\pi}\mathbf B. 
\]
Thus, $\omega$ is symplectic if and only if $\mathbf B$ is non-degenerate (of full rank). 

The associated Bochner Laplacian $\Delta^{L^p}$ is related with the magnetic Laplacian
\[
H^{\hbar}=(i\hbar d+\mathbf A)^*(i\hbar d+\mathbf A), \quad \hbar>0
\]
by the formula
\[
\Delta^{L^p}=\hbar^{-2}H^{\hbar}, \quad \hbar=\frac{1}{p},\quad p\in \NN. 
\]
Let $B : TX\to TX$ be a skew-adjoint endomorphism such that 
\begin{equation}\label{e:BgB}
\mathbf B(u,v)=g(Bu,v), \quad u,v\in TX. 
\end{equation}
Then we have 
\begin{equation}\label{e:JJJ}
J_0=\frac{1}{2\pi}B,\quad J=B(B^*B)^{-1/2}, \quad
\mathcal J=-iB.
\end{equation}
Finally, the function $\tau$ and the constant $\mu_0$ are given by
 \begin{equation}\label{e:tau-mag}
 \tau(x)=\frac 12 \operatorname{Tr}(B^*B)^{1/2}= \operatorname{Tr}^+(B). 
 \end{equation}
 \[
 \mu_0=\inf_{u\in TX}\frac{|(B^*B)^{1/4}u|_g^2}{|u|_g^2}=\inf_{x\in X}\inf |B(x)|.
 \]

\subsection{The 2D case}\label{s:2Dcase}
In this subsection, we compute explicitly the upper bounds of Theorem~\ref{t:upper-estimates} for a two-dimensional magnetic Laplacian and show that, in this case, they are sharp and agree with the asymptotic expansions of \cite{HM01,Helffer-K11}. 

Thus, we assume that $X$ is a Riemann surface and $L$ is trivial. Then one can write
\begin{equation}\label{e:B2}
\mathbf B =b(x)dv_X,
\end{equation}
where $dv_X=\sqrt{g}dx_1\wedge dx_2$ is the Riemannian volume form. If $\mathbf B$ is of full rank, then $b(x)\neq 0$ for all $x\in X$. We will assume $b(x)>0$ for any $x\in X$ and there exists a unique $x_0$ such that 
\[
b(x_0)=b_0:=\min_{x\in X}b(x),
\]

First, we observe that 
\begin{equation}\label{e:tau=b}
\tau(x)=b(x), \quad x\in X. 
\end{equation}

Indeed, choose an orthonormal basis $\{w_1\}$ of $T^{(1,0)}_{x}X$, consisting of eigenvectors of $\mathcal J_{x}$:
\[
\mathcal J_{x}w_1=a_1w_1, 
\]
with some $a_1>0$. Then $e_{1}=\frac{1}{\sqrt{2}}(w_1+\bar w_1)$ and $e_{2}=\frac{i}{\sqrt{2}}(w_1-\bar w_1)$ form an orthonormal basis of $T_{x}X$ such that
\begin{equation}\label{e:mJ-two}
\mathcal J_{x}e_1=-ia_1e_2, \quad \mathcal J_{x}e_2=ia_1e_1.
\end{equation}
By \eqref{e:JJJ}, we get
\[
B_{x}e_1=a_1e_2, \quad B_{x}e_2=-a_1e_1.
\]
and
\begin{equation}\label{e:a1-two}
b(x)=\mathbf B_x(e_1,e_2)=g(B_xe_1,e_2)=a_1. 
\end{equation}
Finally, by \eqref{e:tau-mag}, we complete the proof of \eqref{e:tau=b}:
\[
\tau(x)=\frac 12 \operatorname{Tr}(B_x^*B_x)^{1/2}=b(x).
 \]

Next, we show that the operator $\mathcal F_{1,2,x_0}$ given by \eqref{e:F12} vanishes. 

\begin{lem}\label{l:F12-vanish}
If $x_0$ is a minimum point of $b$, then $\mathcal F_{1,2,x_0}=0$.
\end{lem}

\begin{proof}
First, observe that, in the two-dimensional case, the formula \eqref{e:F12} takes the form
\begin{multline}\label{e:F12-two}
\mathcal F_{1,2,x_0}
= 4 \left \langle R^{TX}_{x_0} \left(\frac{\partial}{\partial z_1}, \frac{\partial}{\partial z_1}\right) \frac{\partial}{\partial \overline{z}_1}, 
\frac{\partial}{\partial \overline{z}_1}\right \rangle\\
+ \left \langle (\nabla ^{X} \nabla ^{X}\mathcal J)_{(\mathcal R,\mathcal R)} \frac{\partial}{\partial z_1},
 \frac{\partial}{\partial \overline{z}_1} \right \rangle 
+\frac{\sqrt{-1}}{4}  \tr_{|TX} \Big(\nabla ^{X} \nabla ^{X}(J\mathcal J)\Big)_{(\mathcal R, \mathcal R)} \\
+ \frac{1}{9} |(\nabla_{\mathcal R}^X \mathcal J) \mathcal R |^2
+ \frac{4}{9} \left\langle(\nabla ^{X}_{\mathcal R} \mathcal J) \mathcal R, \frac{\partial} {\partial z_1}\right\rangle b^+_1 \cL_{x_0}^{-1}b_1 
\left\langle(\nabla ^{X}_{\mathcal R}\mathcal J)\mathcal R,\frac{\partial}
{\partial\overline{z}_1}\right\rangle.
\end{multline}

Recall that we fix an oriented orthonormal basis $\{e_1,e_2\}$ of $T_{x_0}X$ and identify $B^{T_{x_0}X}(0,a_X)$ with $B^{X}(x_0,a_X)$ by the exponential map $\operatorname{exp}^X : T_{x_0}X\to X$, getting a local frame $\{e_1,e_2\}$ on $B^{X}(x_0,a_X)$. We also assume that the identities \eqref{e:mJ-two} hold at $x_0$. 

First, observe that, since the curvature $R^{TX}_{x_0}(u,v)$ is skew-symmetric, the first term in \eqref{e:F12-two} vanishes:
\[
\left \langle R^{TX}_{x_0} \left(\frac{\partial}{\partial z_1}, \frac{\partial}{\partial z_1}\right) \frac{\partial}{\partial \overline{z}_1}, 
\frac{\partial}{\partial \overline{z}_1}\right \rangle=0. 
\]

By \cite[(2.6)]{ma-ma08a}, since $(\nabla ^{X} \nabla ^{X}\mathcal J)_{(\mathcal R,\mathcal R)}$ is skew-adjoint, for the first term in the second line of \eqref{e:F12-two}, we get
\begin{multline}\label{e:412}
\left \langle (\nabla ^{X} \nabla ^{X}\mathcal J)_{(\mathcal R,\mathcal R)} \frac{\partial}{\partial z_1},
 \frac{\partial}{\partial \overline{z}_1} \right \rangle\\ =\frac{i}{2} \left \langle (\nabla ^{X} \nabla ^{X}\mathcal J)_{(\mathcal R,\mathcal R)} e_1, e_2 \right \rangle 
=i\sum_{|\alpha|=2}\partial^\alpha (R^L(e_1,e_2))_{x_0}\frac{Z^\alpha}{\alpha!}\\ +\frac {i}{6} [\langle R^{TX}(\mathcal R, e_1)\mathcal R,\mathcal Je_2\rangle - \langle R^{TX}(\mathcal R, e_2)\mathcal R,\mathcal Je_1\rangle]. 
\end{multline} 
By \eqref{e:RLB} and \eqref{e:B2}, we have
\[
R^L(e_1,e_2)=-i\mathbf B(e_1,e_2)=-ib(Z)\sqrt{\det g(Z)},
\]
where $g_{jk}(Z)=\langle e_j(Z),e_k(Z)\rangle$, $j,k=1,2$. 

Since $(\partial^\alpha b)_{x_0}=0$ for $|\alpha|=1$ and $g_{ij}(0)=\delta_{ij}$, for the first term in \eqref{e:412}, we get
\begin{multline*}
i\sum_{|\alpha|=2}\partial^\alpha (R^L(e_1,e_2))_{x_0}\frac{Z^\alpha}{\alpha!}=\sum_{|\alpha|=2}\partial^\alpha \left(b\sqrt{\det g}\right)_{x_0}\frac{Z^\alpha}{\alpha!}\\ =\sum_{|\alpha|=2}(\partial^\alpha b)_{x_0}\frac{Z^\alpha}{\alpha!}+b_0\sum_{|\alpha|=2}\partial^\alpha \left(\sqrt{\det g}\right)_{x_0}\frac{Z^\alpha}{\alpha!}.
\end{multline*}
By \cite[(1.31)]{ma-ma08a}, we have
\[
\sum_{|\alpha|=2}\partial^\alpha \left(\sqrt{\det g}\right)_{x_0}\frac{Z^\alpha}{\alpha!}=\frac {1}{6} [\langle R^{TX}(\mathcal R, e_1)\mathcal R,e_1\rangle+\langle R^{TX}(\mathcal R, e_2)\mathcal R, e_2\rangle].
\]
So, for the first term in \eqref{e:412}, we get
\begin{multline}\label{e:412-1}
i\sum_{|\alpha|=2}\partial^\alpha (R^L(e_1,e_2))_{x_0}\frac{Z^\alpha}{\alpha!}\\ =\sum_{|\alpha|=2}(\partial^\alpha b)_{x_0}\frac{Z^\alpha}{\alpha!}+\frac {1}{6}b_0[\langle R^{TX}(\mathcal R, e_1)\mathcal R,e_1\rangle+\langle R^{TX}(\mathcal R, e_2)\mathcal R, e_2\rangle].
\end{multline}
For the second term in \eqref{e:412}, by \eqref{e:mJ-two} and \eqref{e:a1-two}, we get 
\begin{multline}\label{e:412-2}
\frac {i}{6} [\langle R^{TX}(\mathcal R, e_1)\mathcal R,\mathcal Je_2\rangle - \langle R^{TX}(\mathcal R, e_2)\mathcal R,\mathcal Je_1\rangle]\\ =-\frac {1}{6}b_0 [\langle R^{TX}(\mathcal R, e_1)\mathcal R, e_1\rangle +\langle R^{TX}(\mathcal R, e_2)\mathcal R, e_2\rangle]
\end{multline}
Combining \eqref{e:412-1} and \eqref{e:412-2}, for the first term in the second line of \eqref{e:F12-two}, we get
\begin{equation}\label{e:413-1}
\left \langle (\nabla ^{X} \nabla ^{X}\mathcal J)_{(\mathcal R,\mathcal R)} \frac{\partial}{\partial z_1},
 \frac{\partial}{\partial \overline{z}_1} \right \rangle = \sum_{|\alpha|=2}(\partial^\alpha b)_{x_0}\frac{Z^\alpha}{\alpha!}. 
\end{equation}

By \eqref{e:def-tau} (see also \cite[(1.94)]{ma-ma08a}), for the second term in the second line of \eqref{e:F12-two}, we have
\begin{equation}\label{e:413-2}
\frac{i}{4}  \tr_{|TX} \Big(\nabla ^{X} \nabla ^{X}(J\mathcal J)\Big)_{(\mathcal R, \mathcal R)}=- \sum_{|\alpha|=2}(\partial^\alpha b)_{x_0}\frac{Z^\alpha}{\alpha!}.
\end{equation}
Combining \eqref{e:413-1} and \eqref{e:413-2}, we infer that the second line in \eqref{e:F12-two} vanishes:
\[
\left \langle (\nabla ^{X} \nabla ^{X}\mathcal J)_{(\mathcal R,\mathcal R)} \frac{\partial}{\partial z_1},
 \frac{\partial}{\partial \overline{z}_1} \right \rangle 
+\frac{\sqrt{-1}}{4}  \tr_{|TX} \Big(\nabla ^{X} \nabla ^{X}(J\mathcal J)\Big)_{(\mathcal R, \mathcal R)}=0.
\]
Finally, the third line in \eqref{e:F12-two} vanishes, since for any $U\in T_{x_0}X$,
\[
\nabla_U^X \mathcal J=0.
\]
Indeed, observe that $\nabla_U^X \mathcal J$ is skew-symmetric, and 
\begin{equation}\label{e:nablaUX}
\langle (\nabla_U^X \mathcal J)e_1, e_2\rangle = U\langle \mathcal Je_1, e_2\rangle-\langle \mathcal J \nabla_U^Xe_1, e_2\rangle-\langle \mathcal J e_1, \nabla_U^Xe_2\rangle. 
\end{equation}
Using \eqref{e:JJJ}, \eqref{e:BgB} and \eqref{e:B2}, we show that the first term in \eqref{e:nablaUX} vanishes:
\[
U\langle \mathcal Je_1, e_2\rangle_{x_0}=-iU[\langle Be_1, e_2\rangle]_{x_0}=-iU[\mathbf B(e_1,e_2)]_{x_0}=-iU[b\sqrt{\det g}]_{x_0}=0.
\]
The last equality holds, because $Ub(x_0)=0$ since $x_0$ is a minimum, and $U[\sqrt{\det g}](x_0)=0$ by properties of normal coordinates. 

By properties of normal coordinates, we have 
\[
(\nabla_U^Xe_1)_{x_0}=(\nabla_U^Xe_2)_{x_0}=0.
\]
Therefore, the last two terms in \eqref{e:nablaUX} vanish. This completes the proof of the lemma.
\end{proof}

Thus, $J_{1,2,x_0}(Z,Z^\prime)=0$ and the model Toeplitz operator $\mathcal T_{x_0}$ in $L^2(T_{x_0}X)$ has the form
\[
\mathcal T_{x_0}=\mathcal P_{x_0}q_{x_0}(Z)\mathcal P_{x_0},
\]
where 
\[
q_{x_0}(Z)=\left(\frac12 {\rm Hess}\,b(x_0)Z,Z\right), \quad Z\in T_{x_0}X.
\]
The spectrum of $\mathcal T_{x_0}$ is computed in \eqref{e:lambda-T}. Therefore, if we denote
\[
a={\rm Tr}\left(\frac12 {\rm Hess}\,b(x_0)\right)^{1/2}, \quad
d=\det \left(\frac12 {\rm Hess}\,b(x_0)\right),
\]
by Theorem~\ref{t:upper-estimates}, we obtain the estimate
\[
\lambda_j(\Delta^{L^p})\leq pb_0
+\left[\frac{2d^{1/2}}{b_0}j+ \frac{a^2}{2b_0}\right]+C_jp^{-1/2}
\]
with some $C_j>0$, which is sharp and agrees with the asymptotic expansions of \cite{HM01,Helffer-K11}.


\begin{thebibliography}{00}
\bibitem{Ali}
Ali, S. T., Englis, M.: Quantization methods: a guide for physicists and analysts. Rev. Math. Phys. 17, 391--490 (2005)

\bibitem{Bargmann} Bargmann, V.: On a Hilbert space of analytic functions and an associated integral transform. Comm. Pure Appl. Math. 14, 187--214 (1961)

\bibitem{Berezin71} Berezin, F.A.: Wick and anti-Wick symbols of operators.  Mat. Sb. (N.S.) 86(128), 578--610 (1971)

\bibitem{Berezin} Berezin, F.A.: General concept of quantization. Commun. Math. Phys. 40, 153–174 (1975)

\bibitem{BMS94}
Bordemann, M., Meinrenken, E., Schlichenmaier, M.: Toeplitz quantization of K\"ahler manifolds and ${\rm gl}(N)$, $N\to\infty$ limits. Comm. Math. Phys. 165, 281--296 (1994)

\bibitem{BPU} 
Borthwick, D., Paul, T., Uribe, A.: Semiclassical spectral estimates for Toeplitz operators. Ann. Inst. Fourier (Grenoble), 48, 1189--1229 (1998)

\bibitem{B-Uribe} Borthwick, D., Uribe, A.:  Almost complex structures and geometric  quantization. Math. Res. Lett. 3, 845--861 (1996)

\bibitem{BG}
Boutet de Monvel, L., Guillemin, V.: The spectral theory of Toeplitz operators. Ann. Math. Studies, Nr. 99, Princeton University Press, Princeton, NJ (1981)

\bibitem{charles03} 
Charles, L.: Berezin-Toeplitz operators, a semi-classical approach. Comm. Math. Phys., 239, 1--28 (2003)

\bibitem{charles03a} 
Charles, L.: Quasimodes and Bohr-Sommerfeld conditions for the Toeplitz operators. Comm. Partial Differential Equations 28, 1527--1566 (2003)

\bibitem{charles06} 
Charles, L.: Symbolic calculus for Toeplitz operators with half-forms 
Journal of Symplectic Geometry 4, 171--198 (2006)

\bibitem{charles}
Charles, L.: Quantization of compact symplectic manifolds. J. Geom. Anal. 26, 2664--2710 (2016)

\bibitem{dai-liu-ma}
Dai, X., Liu, K., Ma, X.: On the asymptotic expansion of Bergman kernel. J. Differential Geom. 72, 1--41 (2006) 

\bibitem{deleporte1}
Deleporte, A: Low-energy spectrum of Toeplitz operators: the case of wells. J. Spectr. Theory 9, 79--125 (2019)

\bibitem{deleporte2}
Deleporte, A.: Low-energy spectrum of Toeplitz operators with a miniwell, Preprint arXiv:1610.05902 (2016)

\bibitem{Englis}
Engli\v s, M.: An excursion into Berezin-Toeplitz quantization and related topics. In: Quantization, PDEs, and geometry, Oper. Theory Adv. Appl., 251, Adv. Partial Differ. Equ. (Basel), pp. 69--115, Birkh\"auser/Springer, Cham (2016)

\bibitem{Gu-Uribe} 
Guillemin, V., Uribe A.:  
The Laplace operator on the $n$th tensor
power of a line bundle: eigenvalues which are uniformly bounded in $n$. 
Asymptotic Anal. 1, 105--113 (1988)

\bibitem{Helffer-K11} 
Helffer, B., Kordyukov, Yu. A.:
Semiclassical spectral asymptotics for a two-dimensional
magnetic Schr\"odinger operator: The case of discrete wells. In:
Spectral Theory and Geometric Analysis; Contemp. Math. 535, pp.
55--78; AMS, Providence, RI (2011)

\bibitem{Helffer-K12} Helffer, B., Kordyukov, Yu. A.: Semiclassical spectral asymptotics for a two-dimensional magnetic Schr\"odinger operator II: The case of degenerate wells. Comm. Partial Differential Equations 37, 1057--1095 (2012)

\bibitem{Helffer-K14}
Helffer, B., Kordyukov, Yu. A.: Semiclassical spectral asymptotics for a magnetic Schr\"odinger operator with non-vanishing magnetic field. In: Geometric Methods in Physics (Bialowieza, Poland, 2013), pp. 259--278, Birkh\"auser, Basel (2014) 

\bibitem{Helffer-K15} Helffer, B., Kordyukov, Yu. A.: Accurate semiclassical spectral asymptotics for a two-dimensional
magnetic Schr\"odinger operator. Ann. Henri Poincar\'e 16, 1651--1688 (2015)

\bibitem{HM01} Helffer, B., Morame, A.:
Magnetic bottles in connection with superconductivity. J.
Funct. Anal.  185, 604--680 (2001) 

\bibitem{HelRo} Helffer, B., Robert, D.: Puits de potentiel g\'en\'eralis\'es et asymptotique semi-classique.  Ann. Inst. H. Poincar\'e Phys. Th\'eor. 41, 291--331 (1984)

\bibitem{HelSj1} Helffer, B., Sj\"ostrand, J.: Multiple wells in the semiclassical limit. I. Comm. Partial Differential Equations 9, 337--408 (1984) 

\bibitem{HM}
Hsiao, C.-Y., Marinescu G.: Berezin-Toeplitz quantization for lower energy forms, Comm.\ Partial Differential Equations. 42, 895--942 (2017)

\bibitem{ioos-lu-ma-ma}
Ioos, L., Lu, W., Ma, X., Marinescu, G.: Berezin-Toeplitz quantization for eigenstates of the Bochner-Laplacian on symplectic manifolds, J Geom Anal (2018). https://doi.org/10.1007/s12220-017-9977-y.

\bibitem{Kor91}
Kordyukov, Yu. A.: $L^p$-theory of elliptic differential operators on 
manifolds of bounded geometry. Acta Appl. Math. 23, 223--260 (1991)

\bibitem{bergman} Kordyukov, Yu. A.: On asymptotic expansions of generalized Bergman kernels on symplectic manifolds. (Russian) Algebra i Analiz 30, no. 2, 163--187 (2018); translation in St. Petersburg Math. J. 30, no. 2, 267--283 (2019)

\bibitem{Ko-ma-ma} Kordyukov, Yu. A., Ma, X., Marinescu, G.: 
Generalized Bergman kernels on symplectic manifolds of bounded geometry. Comm. Partial Differential Equations 44, 1037--1071 (2019)

\bibitem{LeFloch1}
Le Floch, Y.: Singular Bohr-Sommerfeld conditions for 1D Toeplitz operators: elliptic case. Communications in Partial Differential Equations 39, 213--243 (2014)

\bibitem{LeFloch2}
Le Floch, Y.: Singular Bohr-Sommerfeld conditions for 1D Toeplitz operators: hyperbolic case. Anal. PDE 7, 1595--1637 (2014)

\bibitem{LMM}
Lu, W., Ma, X., Marinescu, G.: Donaldson's $Q$-operators for symplectic manifolds. Sci. China Math. 60, 1047--1056 (2017)

\bibitem{ma:ICMtalk} Ma, X.: 
Geometric quantization on {K}\"ahler and symplectic manifolds. In: Proceedings of the International Congress
of Mathematicians. Volume II, pp. 785--810, Hindustan Book Agency, 
New Delhi (2010)

\bibitem{ma-ma02}
Ma, X., Marinescu, G.: The ${\rm Spin}^c$ Dirac operator on high tensor powers of a line bundle. Math. Z. 240, 651--664 (2002) 

\bibitem{ma-ma:book}
Ma, X., Marinescu, G.: Holomorphic Morse inequalities and Bergman kernels. Progress in Mathematics, 254. Birkh\"auser Verlag, Basel (2007) 

\bibitem{ma-ma08a} 
Ma, X., Marinescu, G.: Generalized Bergman kernels on symplectic manifolds. Adv. Math. 217, 1756--1815 (2008)

\bibitem{ma-ma08}
Ma, X., Marinescu, G.: Toeplitz operators on symplectic manifolds. J. Geom. Anal. 18,  565--611 (2008) 
 

\bibitem{RV15} 
Raymond, N., V\~u Ng\d{o}c, S.: Geometry and spectrum in 2D magnetic wells. Ann. Inst. Fourier 65, 137--169 (2015)

\bibitem{Raymond-book}
Raymond, N.: Bound states of the magnetic Schr\"odinger operator. EMS Tracts in Mathematics, 27. European Mathematical Society (EMS), Z\"urich (2017) 

\bibitem{RSIV} 
Reed, M., Simon, B.: Methods of modern mathematical physics. IV. Analysis of operators. Academic Press [Harcourt Brace Jovanovich, Publishers], New York-London (1978)

\bibitem{Schlich10}
Schlichenmaier, M.: Berezin-Toeplitz quantization for compact K\"ahler manifolds. A review of results. Adv. Math. Phys., Art. ID 927280, 38 pp. (2010) 

\bibitem{Simon}
Simon, B.: Semiclassical analysis of low lying eigenvalues. I. Nondegenerate minima: asymptotic expansions. Ann. Inst. H. Poincar\'e Sect. A (N.S.) 38, 295--308 (1983)

\bibitem{Zelditch97} 
Zelditch, S.: 
Index and dynamics of quantized contact transformations. Ann. Inst. Fourier
(Grenoble) 47, 305--363 (1997)
\end{thebibliography}
\end{document}